\documentclass[a4paper]{amsart}
\usepackage{amssymb}
\usepackage{stmaryrd}
\usepackage[all]{xy}
\usepackage{epsf}
\usepackage{enumerate}
\newcommand{\printname}[1] {}

\newtheorem{theorem}{Theorem}[section]
\newtheorem{mtheorem}{Main Theorem}
\newtheorem{mmtheorem}{Main Theorem}
\newtheorem{proposition}[theorem]{Proposition}
\newtheorem{lemma}[theorem]{Lemma}
\newtheorem{corollary}[theorem]{Corollary}

\newtheorem{definition}[theorem]{Definition}
\newtheorem{example}[theorem]{Example}
\theoremstyle{remark}
\newtheorem{remark}{Remark}

\newcommand{\rmap}{\longrightarrow}

\usepackage{amsmath,amscd}
\usepackage{wrapfig}

\begin{document}
\title{A Normal Form Theorem around Symplectic Leaves}
\author{Marius Crainic}
\address{Depart. of Math., Utrecht University, 3508 TA Utrecht,
The Netherlands}
\email{m.crainic@uu.nl}
\author{Ioan M\v arcu\c t}
\address{Depart. of Math., Utrecht University, 3508 TA Utrecht,
The Netherlands}
\email{i.t.marcut@uu.nl}
\begin{abstract}
We prove the Poisson geometric version of the Local Reeb Stability
(from foliation theory) and of the Slice Theorem (from equivariant geometry), which is also
a generalization of Conn's linearization theorem.  
\end{abstract}
\maketitle

\setcounter{tocdepth}{1}

\section*{Introduction}

Recall that a \textbf{Poisson structure} on a manifold $M$ is a Lie bracket $\{\cdot,\cdot\}$ on the space $C^{\infty}(M)$ of smooth functions
on $M$ which acts as a derivation in each entry
\[\{fg,h\}=f\{g,h\}+\{f,h\}g, \quad (\forall) f,g,h \in C^{\infty}(M).\]
A Poisson structure can be given also by a bivector $\pi\in\mathfrak{X}^2(M)$, involutive with respect to the Schouten bracket, i.e.\
$[\pi,\pi]=0$; one has:
\[\langle\pi,df\wedge dg\rangle=\{f,g\},\quad (\forall) f,g \in C^{\infty}(M).\]
To each function $f\in C^{\infty}(M)$ one assigns the Hamiltonian vector field
\[X_f=\{f,\cdot\} \in\mathfrak{X}(M).\]
The flows of the Hamiltonian vector fields give a partition of $M$ into \textbf{symplectic leaves}; they carry
a canonical smooth structure, which makes them into regular immersed submanifolds, whose tangent spaces are spanned by
the Hamiltonian vector fields; each leaf $S$ is a symplectic manifold, with the symplectic structure:
\[\omega_S(X_f,X_g)=\{f,g\}.\]
In this paper we prove a normal form theorem around symplectic leaves, which generalizes Conn's linearization theorem \cite{Conn} (for 1-point
leaves) and is a Poisson geometric analogue of the local Reeb stability from foliation theory and of the slice theorem from group actions. We
will use the Poisson homotopy bundle of a 
leaf $S$, which is the analogue of to the holonomy cover from foliation theory
\[ P_x\rmap S\]
($x\in S$ a base point), whose ``structural group'' is the Poisson homotopy group $G_x$.

\begin{mtheorem} Let $(M, \pi)$ be a Poisson manifold and let $S$ be
a compact leaf. If the Poisson homotopy bundle over $S$ is a smooth compact manifold with vanishing second DeRham cohomology group, then, in a
neighborhood of $S$, $\pi$ is Poisson diffeomorphic to its first order model around $S$.
\end{mtheorem}

A detailed statement and reformulations appear in section \ref{The main theorem again: reformulations and some examples}. The proof uses ideas
similar to the ones in \cite{CrFe-Conn}: a Moser-type argument reduces the problem to a cohomological one (Theorem \ref{theorem-1}); a Van Est
argument and averaging reduces the cohomological problem to an integrability problem (Theorem \ref{TheoremRedInt}) which, in turn, can be
reduced to the existence of special symplectic realization (Theorem \ref{theorem-step-2.2}); the symplectic realization is then built by working
on the Banach manifold of cotangent paths (subsection \ref{Step 2.3: the needed symplectic realization}). For an outline of the paper, we advise
the reader to go through the introductory sentence(s) of each section.

There have been various attempts to generalize Conn's
linearization theorem to arbitrary symplectic leaves. While the desired conclusion was clear (the same as in our theorem), the assumptions (except for the compactness of $S$)
are more subtle. Of course, as for any (first order) local form result, one looks for assumptions on the first
jet of $\pi$ along $S$. Here are a few remarks on the assumptions.\\

$1.$ {\it Compactness assumptions}. It was originally believed that such a result could follow by first applying Conn's theorem to a transversal
to $S$. Hence the expected assumption was, next to the compactness of $S$, that the isotropy Lie algebra $\mathfrak{g}_x$ ($x\in S$) is
semi-simple of compact type. The failure of such a result was already pointed out in \cite{WadeLent}. A refined conjecture was made in
\cite{CrFe0}- revealing the compactness assumptions that appear in our theorem. The idea is the following: while the condition that
$\mathfrak{g}_x$ is semi-simple of compact type is equivalent to the fact that all (connected) Lie groups integrating the Lie algebra
$\mathfrak{g}_x$ are compact, one should require the compactness (and smoothness) of only one group associated to $\mathfrak{g}_x$- the Poisson
homotopy group $G_x$. This is an important difference because
\begin{itemize}
\item our theorem may be applied even when $\mathfrak{g}_x$ is abelian. 
\item actually, under the assumptions of the theorem, $\mathfrak{g}_x$ can be semi-simple of compact type only when the leaf is a point!
\end{itemize}

$2.$ {\it Vanishing of $H^2(P_x)$}. 
The compactness condition on the Poisson homotopy bundle is natural also when drawing an analogy with other local normal form results like local
Reeb stability or the slice theorem. However, compactness alone is not enough (see Example \ref{main-example}). The subtle condition is
$H^2(P_x)=0$ and its appearance is completely new in the context of normal forms:

\begin{itemize}
\item In Conn's theorem, it is not visible (it is automatically satisfied!).
\item In the classical cases (foliations, actions) such a condition is not needed.
\end{itemize}
What happens is that the vanishing condition is related to integrability phenomena \cite{CrFe1, CrFe2}. In contrast with the case of foliations
and of group actions, Poisson manifolds give rise to Lie algebroids that may fail to be integrable. To clarify the role of this assumption, we mention
here:

{\it It implies integrability.}  The main role of this assumption is that it forces the Poisson manifold to be (Hausdorff) integrable around the
leaf. Actually, under such an integrability assumption, the normal form is much easier to establish, and the vanishing condition is not needed-
see our Proposition \ref{main-cor}, which can also be deduced from Zung's linearization theorem \cite{Zung}. Note however that such an
integrability condition refers to the germ of $\pi$ around $S$ (and not the first order jet, as desired!); and, of course, Conn's theorem does
not make such an assumption.

{\it It implies vanishing of the second Poisson cohomology.} Next to integrability, the vanishing condition also implies the vanishing of the
second Poisson cohomology group $H^2_{\pi}(U)$ (of arbitrarily small neighborhoods $U$ of $S$)- which is known to be relevant to infinitesimal
deformations (see e.g. \cite{CrFe0}). We would like to point out that the use of $H^2_{\pi}(U)=0$ only simplifies our argument but is not essential.
A careful analysis shows that one only needs a certain class in $H^2_{\pi}(U)$ to vanish, and this can be shown using only integrability. 
This is explained at the end of subsection \ref{Step 2.1: Reduction to integrability}, when concluding the proof of Proposition
\ref{main-cor} mentioned above.\\

\vskip 5pt
\noindent \textbf{Acknowledgments.} We would like to thank Rui Loja Fernandes, David Martinez Torres and Ezra Getzler for their very useful comments.
This research was supported by the NWO Vidi Project ``Poisson topology''.

\section{A more detailed introduction}

In this section we give more details on the statement of the main theorem. We start by recalling some classical normal form theorems in
differential geometry. Then we discuss the local model associated to a principal bundle over a symplectic manifold. Next we describe the Poisson
homotopy bundle in detail and we finish the section with an overview of the notion of integrability of Poisson manifolds.

\subsection{The Slice Theorem}
\label{The Slice Theorem}

Let $G$ be Lie group acting on a manifold $M$, $x\in M$ and denote by $\mathcal{O}$ the orbit through $x$. The Slice Theorem (\cite{DK}) gives a
normal form for the $G$-manifold $M$ around $\mathcal{O}$. It is built out of the isotropy group $G_x$ at $x$ and its canonical representation
$V_x= T_xM/T_x\mathcal{O}$. Explicitly, the local model is:
\[ G\times_{G_x} V_x= (G\times V_x)/G_{x}\]
which is a $G$-manifold and admits $\mathcal{O}$ as the orbit corresponding to $0\in V_x$.

\begin{theorem}
If $G$ is compact, then a $G$-invariant neighborhood of $\mathcal{O}$ in $M$ is diffeomorphic, as a $G$-manifold, to a $G$-invariant
neighborhood of $\mathcal{O}$ in $G\times_{G_x} V_x$.
\end{theorem}

It is instructive to think of the building pieces of the local model as a triple $(G_x, G\rmap \mathcal{O}, V_x)$ consisting of the Lie group
$G_x$, the principal $G_x$-bundle $G$ over $\mathcal{O}$ and a representation $V_x$ of $G_x$. This triple should be thought of as the first
order data (first jet) along $\mathcal{O}$ associated to the $G$-manifold $M$, while of the associated local model as the first order
approximation.

\subsection{Local Reeb stability}

Let $\mathcal{F}$ be a foliation on a manifold $M$, $x\in M$ and denote by $L$ the leaf through $x$. The Local Reeb Stability Theorem
(\cite{MM}) gives a normal form for the foliation around $L$ (we state below a weaker version). Denote by $\widetilde{L}$ the universal cover of
$L$, and consider the linear holonomy representation of $\Gamma_x:=\pi_1(L,x)$ on $N_x= T_xM/T_xL$. The local model is
$\widetilde{L}\times_{\Gamma_x} N_x$
 with leaves $\widetilde{L}\times_{\Gamma_x} (\Gamma_xv)$ for $v\in N_x$; $L$ corresponds to $v=0$.

\begin{theorem}
If $L$ is compact and $\Gamma_x$ is finite, then a saturated neighborhood of $L$ in $M$ is diffeomorphic, as a foliated manifold, to a
neighborhood of $L$ in $\widetilde{L}\times_{\Gamma_x} N_x$.
\end{theorem}

Again, the local model is build out of a triple $(\Gamma_x, \widetilde{L}\rmap L, N_x)$, consisting of the discrete group $\Gamma_x$, the
principal $\Gamma_x$-bundle $\widetilde{L}$ and a representation $N_x$ of $\Gamma_x$. The triple should be thought of as the first order data
along $L$ associated to the foliated manifold $M$ and the local model as the first order approximation.

\subsection{Conn's Linearization Theorem}
\label{Conn's theorem}

Let $(M,\pi)$ be a Poisson manifold and $x\in M$ be a zero of $\pi$. Conn's theorem \cite{Conn} gives a normal form for $(M, \pi)$
near $x$, built out of the isotropy Lie algebra $\mathfrak{g}_x$. Recall that $\mathfrak{g}_x=
T^*_{x}M$ with the bracket:
\begin{equation}\label{isotr}
[d_xf,d_xg]=d_x \{ f,g\},\ \ \ \   f,g \in C^{\infty}(M).
\end{equation}
Conversely, there is a Poisson bracket $\pi_{\mathrm{lin}}$ on the dual $\mathfrak{g}^*$ of any Lie algebra:
\[\pi_{\mathrm{lin}}(X,Y)_{\xi}:=\langle\xi,[X,Y]\rangle, \ (\forall) \xi\in\mathfrak{g}^*, \ X,Y\in\mathfrak{g}=T^*_{\xi}\mathfrak{g}^*.\]
\begin{theorem}
If $\mathfrak{g}_x$ is semi-simple of compact type then a neighborhood of $x$ in $M$ is Poisson-diffeomorphic to a neighborhood of the origin in
$(\mathfrak{g}_{x}^{*},\pi_{\mathrm{lin}})$.
\end{theorem}

Again, the local data (the Lie algebra $\mathfrak{g}_x$) should be viewed as the first order data at $x$ associated to the Poisson manifold,
while the local model as the first order approximation. To make the analogy with the previous two theorems, we replace the Lie
algebra $\mathfrak{g}_x$ by $G_x:=G(\mathfrak{g}_x)$, the $1$-connected Lie group integrating $\mathfrak{g}_x$. The local data is then
\[ (G_x, G_x\rmap \{x\}, \mathfrak{g}_{x}^{*} )\]
and the local model is defined on $G_{x}\times_{G_x} \mathfrak{g}_{x}^{*}= \mathfrak{g}_{x}^{*}$. The most convincing argument for bringing
$G_x$ into the picture is the fact that the assumption of the theorem is equivalent to the fact that $G_x$ is compact.

\subsection{The local model}
\label{The local model}
In this subsection we explain the local model. The construction given below is standard in symplectic geometry and goes back to the local forms
of Hamiltonian spaces around the level sets of the moment map (cf. e.g. \cite{GS}) and also shows up in the work of Montgomery \cite{Montgomery}
(see also \cite{SymplecticFibrations}).

The starting data is a again a triple consisting of a symplectic manifold $(S,\omega_S)$, which will be our symplectic leaf, a principal
$G$-bundle $P$ over $S$, which will be the Poisson homotopy bundle and the coadjoint action of $G$ on $\mathfrak{g}^*$:
\[(G,P\rmap (S,\omega_S),\mathfrak{g}^*).\]
As before, $G$ acts diagonally on $P\times\mathfrak{g}^*$. As a manifold, the local model is:
\[ P\times_{G}\mathfrak{g}^*= (P\times \mathfrak{g}^*)/G.\]
To describe the Poisson structure, we choose a connection 1-form on $P$, $\theta\in \Omega^1(P, \mathfrak{g})$.
The $G$-equivariance of $\theta$ implies that the $1$-form $\tilde{\theta}$ on $P\times \mathfrak{g}^*$ defined by
\[\tilde{\theta}_{(p,\mu)}=\langle\mu,\theta_p\rangle\]
is $G$-invariant. Consider now the $G$-invariant 2-form
\[\Omega:=p^*(\omega_S)-d\tilde{\theta}\in \Omega^2(P\times \mathfrak{g}^*).\]
The open set $M\subset P\times \mathfrak{g}^*$ where it is non-degenerate contains $P\times\{0\}$, therefore $(M,\Omega)$ is a symplectic
manifold on which $G$ acts freely, in a Hamiltonian fashion, with moment map given by the second projection. Hence $N= M/G\subset P\times_G
\mathfrak{g}^*$ inherits a Poisson structure $\pi_N$. Notice that $S$ sits as a symplectic leaf in $(N, \pi_N)$:
\[(S,\omega_S)=(P\times \{0\},\Omega_{|P\times \{0\}})/G.\]
\begin{definition}\label{definition_Poisson_neighborhood}
A \textbf{Poisson neighborhood of $S$ in $P\times_{G}\mathfrak{g}^{*}$} is any Poisson structure of the type just described, defined on a
neighborhood $N$ of $S$.
\end{definition}
Note that different connections induce Poisson structures which have Poisson-diffeomorphic open neighborhoods of $S$. Also, intuitively, $\pi_N$
is constructed by combing the canonical Poisson structure on $\mathfrak{g}^*$ with the pullback of the symplectic structure $\omega_S$ to $P$.
For instance, if $P$ is trivial, using the canonical connection the resulting model is $(S, \omega)\times (\mathfrak{g}^*,\pi_{\mathrm{lin}})$.
Note also that, when $P$ is compact, one can find small enough Poisson neighborhoods of type:
\[ N= P\times_{G} V\]
with $V\subset \mathfrak{g}^*$ a $G$-invariant open containing $0$. Moreover, the resulting symplectic leaves do not depend, as manifolds, on
$\theta$. Denoting by $\mathcal{O}_{\xi}:=G\xi$, these are:
\[ P/G_{\xi}\cong P\times_{G}\mathcal{O}_{\xi} \subset P\times_{G} V, \ \ \  \xi\in V.\]

\begin{example}\label{torus-bundle}\rm \ To understand the role of the bundle $P$ it is instructive to look
at the case when $G= T^q$ is a $q$-torus. As a foliated manifold, the local model is:
\[ P\times_{G} \mathfrak{g}^*= S\times \mathbb{R}^q= \bigcup_{t} S\times \{t\} \ \ (t= (t_1, \ldots , t_q)\in \mathbb{R}^q).\]
To complete the description of the local model as a Poisson manifold, we need to specify the symplectic forms $\omega_t$ on $S$- and this is
where $P$ comes in. Principal $T^q$-bundles are classified by $q$ integral cohomology classes $c_1, \ldots , c_q\in H^2(S)$; the choice of the
connection $\theta$ above corresponds to the choice of representatives $\omega_1, \ldots , \omega_q\in \Omega^2(S)$ and the resulting Poisson
structure corresponds to
\[ \omega_{t}= \omega_{S}+ t_1\omega_1+ \ldots + t_q\omega_q.\]
\end{example}

\begin{remark}\label{remark-Dirac}\rm
Dirac geometry (see \cite{BR} for the basic definitions) provides further insight into our construction. Recall that one of the main features of
Dirac structures is that, although they generalize closed 2-forms, they can be pushed-forward. In particular, our 2-form $\Omega$ can be
pushed-forward to give a Dirac structure $L(\theta)$ on the entire space $P\times_{G}\mathfrak{g}^*$. Another feature of Dirac structure is that
they generalize Poisson bivectors; actually, for a general Dirac structure $L\subset TM\oplus T^*M$ on $M$, one can 
talk about the largest open on which $L$ is Poisson (Poisson support of $L$):
\[ \textrm{sup}(L):= \{x\in M: pr_2(L_x)= T_{x}^{*}M\}.\]
Our local model arises from the fact that $S$ is inside the support of $L(\theta)$. Also the independence of $\theta$ fits well in this context:
if $\theta'$ is another connection, then $L(\theta')$ is the gauge transform of $L(\theta)$ with respect to
$d(\widetilde{\theta}-\widetilde{\theta'})$. A simple version of Moser's Lemma can be used to show that $L(\theta)$ and $L(\theta')$ are
isomorphic around $S$.
\end{remark}

\subsection{The Poisson homotopy bundle I: via cotangent paths}
\label{The Poisson homotopy bundle over a symplectic leaf}
For the statement of the main theorem, we still have to discuss the Poisson homotopy bundle over a symplectic leaf.
In this subsection we provide a first description, completely analogous to the construction of the universal cover of a manifold.
It is based on the idea that Poisson geometry is governed by ``contravariant geometry''- for which we use \cite{CrFe2,Fernandes}
as references. We recall here a few basic facts. The overall idea is that, in Poisson geometry, the relevant directions are the
``cotangent ones'', i.e., for a Poisson manifold $(M, \pi)$, one should replace the tangent bundle $TM$ by the cotangent one $T^*M$. The
two are related by
\[ \pi^{\sharp}: T^*M\longrightarrow TM , \ \pi^{\sharp}(\alpha)(\beta)= \pi(\alpha, \beta).\]
Of course, $T^*M$ should be considered together with the structure that allows us to treat it as a ``generalized tangent bundle'', i.e. with its
canonical structure of Lie algebroid: the anchor is $\pi^{\sharp}$, while the Lie bracket given by
\begin{equation}\label{bracket_cotangent}
[\alpha, \beta]_{\pi}= L_{\pi^{\sharp}(\alpha)}(\beta)-L_{\pi^{\sharp}(\beta)}(\alpha)-d\pi(\alpha,\beta),\quad (\forall)\alpha,\beta\in \Gamma(T^*M).
\end{equation}
According to this philosophy,  the analogue of the universal cover of a manifold should use ``cotangent paths'' instead of paths. Recall that a
\textbf{cotangent path} in $(M, \pi)$ is a path $a:[0,1]\rmap T^*M$, above some path $\gamma: [0, 1]\longrightarrow M$, such that
\[ \pi^{\sharp}(a(t))=\frac{d}{dt} \gamma(t) .\]
Similarly, one can talk about cotangent homotopies and one defines \textbf{the Poisson homotopy groupoid} of $(M, \pi)$, denoted $\Sigma(M,
\pi)$ (also called the Weinstein groupoid \cite{CrFe1, CrFe2}), as the space consisting of cotangent homotopy classes of paths:
\[ \Sigma(M,\pi)= \frac{\textrm{cotangent\ paths}}{\textrm{cotangent\ homotopy}}.\]
The source/target maps $s, t: \Sigma(M,\pi)\rmap M$ take a cotangent path into the initial/final point of the base path.

\begin{definition}
The \textbf{Poisson homotopy bundle} of $(M, \pi)$ at $x$ is $P_x:= s^{-1}(x)$ (the set of cotangent homotopy classes of cotangent paths starting at $x$).
\end{definition}

Recall that $\Sigma(M, \pi)$ is a groupoid, where the composition is given by concatenation of cotangent paths (here, to stay within the class
of smooth paths, the concatenation is slightly perturbed using a bump function; however, up to cotangent homotopy, the result does not depend on
the choice of the bump function- again, see \cite{CrFe2} for details). In particular, $G_x:= s^{-1}(x)\cap t^{-1}(x)$ is a group, which we will
call the \textbf{Poisson homotopy} group of $(M, \pi)$ at $x$. Also, the composition defines a free action of $G_x$ on $P_x$, and the quotient
is identified with the symplectic leaf $S_x$ through $x$, via the target map
\[ P_x\longrightarrow S_x,\ \ [a]\mapsto \gamma(1).\]

Regarding the smoothness of $\Sigma(M,\pi)$, one remarks that it is a quotient of the (Banach) manifold of cotangent paths of class $C^1$. We
are interested only in smooth structures which make the corresponding quotient map into a submersion. Of course, there is at most one such
smooth structure on $\Sigma(M, \pi)$; when it exists, one says that \textbf{$\Sigma(M, \pi)$ is smooth}, or that $(M, \pi)$ is
\textbf{integrable}. Note that in this case $\Sigma(M, \pi)$ will be a finite dimensional manifold, but which may fail to be Hausdorff. If also
the Hausdorffness condition is satisfied, we say that $(M, \pi)$ is \textbf{Hausdorff integrable}. Completely analogously, one makes sense of
the smoothness of $P_x$ and of $G_x$. Note however that, whenever smooth, these two will be automatically Hausdorff. Moreover, the smoothness of
$P_x$ is equivalent to that of $G_x$- and this is controlled by the \textbf{monodromy map} at $x$, which is a group homomorphism
\begin{equation}
\label{partial}
\partial: \pi_2(S)\rmap G(\mathfrak{g}_x)
\end{equation}
into the 1-connected Lie group $G(\mathfrak{g}_x)$ integrating $\mathfrak{g}_x$. Intuitively, $\partial$ encodes the variation of symplectic
areas, while in the smooth case, $\partial$ can also be identified with the boundary in the homotopy long exact sequence associated to $P_x\rmap
S_x$. From \cite{CrFe2}, we mention here:

\begin{proposition}\label{smoothness-all-in-one} The Poisson homotopy bundle $P_x$ at $x$ is smooth if and only if the image of $\partial_x$ is a discrete subgroup of $G(\mathfrak{g}_x)$.

In this case $P_x$ is smooth principal $G_x$-bundle over $S_x$, the Lie algebra of $G_x$ is $\mathfrak{g}_x$, $\pi_0(G_x)\cong \pi_1(S)$ and
identity component $G_{x}^{0}$ is isomorphic to $G(\mathfrak{g}_x)/Im(\partial_x)$.
\end{proposition}

Coming back to our normal forms:

\begin{definition}\label{first-order-jet-def1} Assuming that $P_x$ is smooth, \textbf{the first order local model} of $(M,\pi)$ around $S= S_x$
is defined as the local model (in the sense of subsection \ref{The local model}) associated to the Poisson homotopy bundle. The Poisson
structure on the local model (well defined up to diffeomorphisms) is denoted $j^{1}_{S}\pi$ and is called \textbf{the first order approximation
of $\pi$ along $S$}.
\end{definition}

The fact that the Poisson homotopy bundle encodes the first jet of $\pi$ along $S$ will be explained in the next subsection; the fact that
$j^{1}_{S}\pi$ deserves the name of first order approximation of $\pi$ along $S$ is explained in section \ref{Poisson structures around a
symplectic leaf: the algebraic framework} (subsection \ref{The dilatation operators and jets along $S$}).



\subsection{The Poisson homotopy bundle II: via its Atiyah sequence}
\label{subsection-Atiyah sequences}

In this subsection we present a slightly different point of view on the Poisson homotopy bundle $P_x$. The main remark is that the
$P_x$ is not visible right away as a smooth principal bundle, but through its
infinitesimal data, i.e. an ``abstract Atiyah sequence''.

This point of view has several advantages. For instance, it will allow us to see that, indeed, $P_x$ encodes the first order jet of $\pi$ along
the symplectic leaf $S_x$. It also implies that the local model can be constructed without the smoothness assumption on $P_x$. Also, this
approach to $P_x$ does not really need the use of \cite{CrFe1} (integrability of Lie algebroids); actually, (abstract) Atiyah sequences appeared
outside the theory of Lie algebroids, as the infinitesimal counterparts of principal bundles. However, we will appeal to the language of Lie
algebroids, as it simplifies the discussion. In particular, an \textbf{abstract Atiyah sequence} over a manifold $S$ is simply a transitive Lie
algebroid $A$ over $S$, thought of as the exact sequence of Lie algebroids:
\begin{equation}
\label{abstract-Atiyah}
0\rmap Ker(\rho)\rmap A\stackrel{\rho}{\rmap} TS\rmap 0,
\end{equation}
where $\rho$ is the anchor map of $A$. Any principal $G$-bundle $p: P\longrightarrow S$ gives rise to such a sequence,
known as \textbf{the Atiyah sequence} associated to $P$:
\[ 0\longrightarrow P\times_{G} \mathfrak{g}^*\longrightarrow TP/G\stackrel{(dp)}{\longrightarrow} TS \rmap 0.\]
Here, the Lie algebroid is $A(P):= TP/G$ and the bracket on $\Gamma(A(P))$ comes from the Lie bracket of $G$-invariant vector fields on $P$, via
the identification
\[ \Gamma(A(P))= \mathfrak{X}(P)^G.\]

Given an abstract Atiyah sequence (\ref{abstract-Atiyah}) over $S$, one says that it is \textbf{integrable} if there exists a principal
$G$-bundle $P$ (for some Lie group $G$) such that $A$ is isomorphic to $A(P)$; one also says that $P$ integrates (\ref{abstract-Atiyah}). This
notion was already considered in \cite{Almeida} without any reference to Lie algebroids. However, it is clear that this condition is equivalent
to the integrability of $A$ as a transitive Lie algebroid (see also \cite{Mackenzie, MackenzieLL,CrFe1}). In particular, as for Lie groups, if
(\ref{abstract-Atiyah}) is integrable, then there exists a unique (up to isomorphism) 1-connected principal bundle integrating it.

\begin{remark}
\label{gauge-groupoids} The fact that abstract Atiyah sequences are the infinitesimal counterparts of principal bundles also follows from the
fact that transitive groupoids are essentially the same thing as principal bundles: any principal $G$-bundle $p: P\rmap S$ induces a transitive
Lie groupoid over $S$- the quotient of the pair groupoid of $P$ modulo the diagonal action of $P$ (called the gauge groupoid of $P$);
conversely, any transitive Lie groupoid $\mathcal{G}$ over $S$ arises in this way- just choose $x\in S$ and choose $t: P_x= s^{-1}(x)\rmap S$,
with structural group the isotropy group $G_x= s^{-1}(x)\cap t^{-1}(x)$.
\end{remark}

Back to our Poisson manifold $(M, \pi)$, one has an abstract Atiyah sequence
\begin{equation}
\label{abstract-Atiyah-leaf}
0\rmap \nu_{S}^{*} \longrightarrow T^{*}_{S}M\stackrel{\pi^{\sharp}}{\longrightarrow} TS\rmap 0 .
\end{equation}
above any symplectic leaf $S= S_x$. Of course, $T^{*}_{S}M$ is just the restriction to $S$ of the cotangent Lie algebroid (see the
previous subsection). The description of $P_x$ in terms of paths can also be seen as the general construction of \cite{CrFe1}
applied to this Lie algebroid. We conclude (using the above mentioned references):

\begin{proposition} Given $x\in S$, the Poisson homotopy bundle $P_x$ is smooth if and only if the abstract Atiyah sequence
(\ref{abstract-Atiyah-leaf}) is integrable. Moreover, in this case $P_x$ is the unique integration of (\ref{abstract-Atiyah-leaf}) which is
1-connected.
\end{proposition}

Next, we show that the abstract Atiyah sequence (\ref{abstract-Atiyah-leaf}) (hence also the Poisson homotopy bundle) encodes the first jet of
$\pi$ along $S$. Consider the differential graded Lie algebra $\mathfrak{X}^{\bullet}(M)$ of multivector fields on $M$, the sub-algebra
$\mathfrak{X}^{\bullet}_{S}(M)$ consisting of multivector fields whose restriction to $S$ belongs to $\mathfrak{X}^{\bullet}(S)$ and the square
$I^2(S)\subset C^{\infty}(M)$ of the ideal $I(S)$ of smooth functions which vanish on $S$. The first order jets along $S$ are controlled by the
map of graded Lie algebras
\[ j_{S}^{1}: \mathfrak{X}^{\bullet}_{S}(M)\longrightarrow \mathfrak{X}^{\bullet}_{S}(M)/I^2(S)\mathfrak{X}^{\bullet}(M).\]
We see that, for the symplectic $(S,\omega_S)$, first jets along $S$ of Poisson structures that have $(S,\omega_S)$ as a symplectic leaf
correspond to elements
\begin{equation*} 
\tau\in \mathfrak{X}^{2}_{S}(M)/I^2(S)\mathfrak{X}^{2}(M)\ \ \ \textrm{satisfying}\ \ [\tau, \tau]= 0
\end{equation*}
and with the property that the restriction map
\begin{equation*}
r_{S}: \mathfrak{X}^{2}_{S}(M)/I(S)^{2}\mathfrak{X}^{2}(M)\longrightarrow \mathfrak{X}^2(S),\ \ \textrm{sends}\ \tau \ \textrm{to} \ \omega_{S}^{-1}.
\end{equation*}
We denote by $J^{1}_{(S, \omega_s)}\textrm{Poiss}(M)$ the set of such elements $\tau$. It is interesting that any such $\tau$ comes from a
Poisson structure $\pi$ defined on a neighborhood of $S$ in $M$ (this follows from the discussion below). Note that, starting with $(S,
\omega_S)$, one always has a short exact sequence of vector bundles over $S$:
\begin{equation}
\label{pre-Atiyah} 0\rmap \nu_{S}^{*}\rmap T^{*}_{S}M
\stackrel{\rho_{\omega_S}}{\rmap} TS \rmap 0, \ \ \  \textrm{with}\ \ \  \rho_{\omega_S}(\xi)=\omega_S^{-1}(\xi_{|TS}).
\end{equation}

\begin{proposition}\label{1st-jet-Atiyah} Given a submanifold $S$ of $M$ and a symplectic form $\omega_{S}$ on $S$, there is a 1-1 correspondence between elements $\tau\in J^{1}_{(S, \omega_S)}\textrm{Poiss}(M)$
and Lie brackets $[\cdot, \cdot]_{\tau}$ on $T^{*}_{S}M$ making (\ref{pre-Atiyah}) into an abstract Atiyah sequence.
\end{proposition}

\begin{remark}\label{non-smooth-model}
The local model (from our main theorem) can be described without using the smoothness $P_x$. This was explained by Vorobjev (see
\cite{Vorobjev,Vorobjev2} and our section \ref{Poisson structures around a symplectic leaf: the algebraic framework}). Here we indicate a
different approach using Dirac structures. We start with a symplectic manifold $(S, \omega_S)$ and an abstract Atiyah sequence over $S$
\[ 0 \rmap K\rmap A\stackrel{\rho}{\rmap} TS \rmap 0,\]
and we construct a Poisson structure around the zero section of the dual $K^*$ of the vector bundle $K$, generalizing the integrable case (note
that, if $A\cong A(P)$, then $K^*= P\times_{G}\mathfrak{g}^*$). Choose a splitting $\theta: A\rmap K$. As for Lie algebras, one has a
``fiberwise linear Poisson structure'' $\pi_{\textrm{lin}}$ on $A^{*}$. Using $\theta^*: K^*\rmap A^*$, we can pull-back $\pi_{\textrm{lin}}$ to
a Dirac structure $L_\theta$ on $K^*$. Let $\omega\in \Omega^2(K^*)$ be the pull-back of $\omega_S$ along the projection. The gauge transform of
$L_\theta$ of with respect to $\omega$, $L(\theta):= L_\theta^{\omega}$, gives the local model in full generality (here, again, $S\subset
supp(L(\theta)$). Indeed, one checks that, when $A$ comes from a principal bundle $P$, this is precisely the Dirac structure mentioned in Remark
\ref{remark-Dirac}); for full details, see \cite{teza}.
\end{remark}


\subsection{More on integrability}
\label{The Poisson homotopy bundle III: via  symplectic groupoids}
There are a few more aspects of $\Sigma(M, \pi)$ that deserve to be recalled here. We use as references
\cite{CrFe1, CaFe, CrFe2, MackXu, MW}.

First of all, while $\Sigma(M, \pi)$ uses the cotangent Lie algebroid $T^*M$ of the Poisson manifold, a similar construction applies to any
Lie algebroid $A$. The outcome is a topological groupoid $\mathcal{G}(A)$ whose smoothness if equivalent to the
integrability of $A$. For instance, when $A= \mathfrak{g}$ is a Lie algebra,
$\mathcal{G}(\mathfrak{g})$ is the unique 1-connected Lie group with Lie algebra $\mathfrak{g}$. When $A= TM$, then $\mathcal{G}(TM)$ is the usual
homotopy groupoid of $M$. Implicit in the previous discussion above is the fact that, for a symplectic leaf $S$ of a Poisson manifold $(M, \pi)$,
$\mathcal{G}(T^*M|_{S})= \mathcal{G}(T^*M)|_{S}$ is encoded by the Poisson homotopy bundle (see Remark \ref{gauge-groupoids}).

The second point is that, while $\mathcal{G}(A)$ makes sense for any Lie algebroid $A$, in the Poisson case, $\mathcal{G}(T^*M)= \Sigma(M, \pi)$
is a symplectic groupoid, i.e, comes endowed with a symplectic structure $\omega$ which is compatible with the groupoid composition (i.e. is
multiplicative) .  There are two uniqueness phenomena here:
\begin{itemize}
\item $\omega$ is the unique multiplicative form for which $s: \Sigma\rmap M$ is a Poisson map.
\item For a general symplectic groupoid $(\Sigma, \omega)$ over a manifold $M$, there is a unique Poisson structure $\pi$ on $M$ such that $s: \Sigma\rmap M$ is a Poisson map.
\end{itemize}
Combining with Lie II for Lie algebroids we deduce that, for a Poisson manifold $(M, \pi)$, if $(\Sigma, \omega)$ is a symplectic Lie groupoid
with 1-connected $s$-fibers, and if $s: (\Sigma, \omega)\rmap (M,\pi)$ is a Poisson map, then $\Sigma$ is isomorphic to $\Sigma(M,\pi)$. For us,
this gives a way of computing Poisson homotopy bundles more directly.

\begin{example} \label{P-bdle-ex1}\rm
Consider the linear Poisson structure $(\mathfrak{g}^*,\pi_{\mathrm{lin}})$ on the dual of a Lie algebra $\mathfrak{g}$. Let $G=
G(\mathfrak{g})$. We have that $T^*G$, endowed with the canonical symplectic structure, is the $s$-fiber 1-connected symplectic groupoid
integrating $\pi_{\mathrm{lin}}$. Using the identifying $T^*G= G\times \mathfrak{g}^*$, given by left translations, the groupoid structure is
that of the action groupoid of $G$ on $\mathfrak{g}^*$. It follows that, for any $\xi\in \mathfrak{g}^*$, the symplectic leaf through $\xi$ is
the coadjoint orbit $\mathcal{O}_{\xi}$ and the associated Poisson homotopy bundle is precisely the $G_{\xi}$-bundle $G\rmap \mathcal{O}_{\xi}$.
\end{example}

\section{The main theorem again: reformulations and some examples}
\label{The main theorem again: reformulations and some examples}
In this section we give a complete statement the main theorem, two equivalent formulations and several examples.

As a summary of the previous section: given $(M, \pi)$ and the symplectic leaf $S$ through $x$, the first order jet of $\pi$ along $S$ is
encoded in the Poisson homotopy bundle $P_x\rmap S$; out of it we produced  the local model $P_x\times_{G_x}\mathfrak{g}_{x}^{*}$ which, around
$S$, is a Poisson manifold admitting $S$ as a symplectic leaf.

\begin{mmtheorem}[complete version] \label{Main-Theorem-2p}
Let $(M, \pi)$ be a Poisson manifold, $S$ a compact symplectic leaf and $x\in S$. If $P_x$, the Poisson homotopy bundle at $x$, is smooth,
compact, with
\begin{equation}\label{subtle-cond}
H^2(P_x; \mathbb{R})= 0,
\end{equation}
then there exists a Poisson diffeomorphism between an open neighborhood of $S$ in $M$ and a Poisson neighborhood of $S$ in the local model
$P_x\times_{G_x}\mathfrak{g}_{x}^{*}$ associated to $P_x$, which is the identity on $S$.
\end{mmtheorem}
\begin{remark}
The open neighborhood of $S$ in $M$ can be chosen to be saturated. Indeed, by the comments following Definition
\ref{definition_Poisson_neighborhood}, the open in the local model can be chosen of the type $P_x\times_{G_x}V$, and this is a union of compact
symplectic leaves.
\end{remark}
Comparing with the classical results from foliation theory and group actions, the surprising condition is (\ref{subtle-cond}). As we shall soon
see, this condition is indeed necessary and it is related to integrability. However, as the next proposition shows, this condition is
not needed in the Hausdorff integrable case. 

\begin{proposition}\label{main-cor}
In the main theorem, if $S$ admits a neighborhood which is Hausdorff integrable, then the assumption (\ref{subtle-cond}) can be dropped.
\end{proposition}

Note also that, in contrast with the proposition, if $M$ is compact, then the 
conditions of our main theorem cannot hold {\it at all} points $x\in M$ (since it would follow that $\Sigma(M, \pi)$ is compact
and its symplectic form is exact- see \cite{CrFe0}).

Next, since the conditions of the theorem may be difficult to check in explicit examples, we reformulate them in terms of the Poisson
homotopy group $G_x$ at $x$.

\begin{proposition}\label{Prop_restatement_1} The conditions of the main theorem are equivalent to:
\begin{enumerate}
\item[(1)] The leaf $S$ is compact.
\item[(2)] The Poisson homotopy group $G_x$ is smooth and compact.
\item[(3)] The dimension of the center of $G_x$ equals to the rank of $\pi_2(S, x)$.
\end{enumerate}
\end{proposition}

\begin{proof} We already know that the smoothness of $P_x$ is equivalent to that of $G_x$ while, under this
smoothness condition, compactness of $P_x$ is clearly equivalent to that of $S$ and $G_x$. Hence, assuming (1) and (2), we still have to
show that (3) is equivalent to $H^2(P_x)=0$. Since $G_x$ is compact,
$\mathfrak{g}_x$ is a product of a semi-simple Lie algebra $\mathfrak{H}$ of compact type with its center $\zeta$.
Therefore $G(\mathfrak{g}_x)= H\times \zeta$, with $H$ compact 1-connected, and $\partial_x$ takes values in $Z\times \zeta$, where $Z= Z(H)$ is
a finite group. Since $\pi_2(P_x)$ can be identified with $\mathrm{Ker}(\partial_x)$ we have an exact sequence
\[ 0\rmap \pi_2(P_x)\otimes_{\mathbb{Z}}\mathbb{R}\rmap \pi_2(S)\otimes_{\mathbb{Z}}\mathbb{R}\stackrel{\partial_{\mathbb{R}}}{\rmap} \zeta \]
where, since $P_x$ is 1-connected, the first term is canonically isomorphic to $H_2(P_x; \mathbb{R})$. Finally, since the connected component of
the identity in $G_x$ is $(H\times \zeta)/Im(\partial_x)$, its compactness implies that $\partial_{\mathbb{R}}$ is surjective. Hence a short
exact sequence
\[ 0\rmap H_2(P_x)\rmap \pi_2(S)\otimes_{\mathbb{Z}}\mathbb{R}\stackrel{\partial_{\mathbb{R}}}{\rmap} \zeta \rmap 0.\]
Therefore, condition (3) is equivalent to the vanishing of $H_2(P_x)$.
\end{proof}

Next, using the monodromy group, one can also get rid of the $G_x$. Recall \cite{CrFe2, CrFe1} that the monodromy group of $(M, \pi)$ at $x$,
denoted by $\mathcal{N}_x$, is the image of $\partial_x$ (see (\ref{partial})) intersected with the connected component of the center of
$G(\mathfrak{g}_x)$. Using the exponential it can be viewed as a subgroup of the center of $\mathfrak{g}_x$
\[ \mathcal{N}_x\subset Z(\mathfrak{g}_x).\]

\begin{proposition}\label{Prop_restatement_2} The conditions of the main theorem are equivalent to:
\begin{enumerate}
\item The leaf $S$ is compact with finite fundamental group.
\item The isotropy Lie algebra $\mathfrak{g}_x$ is of compact type.
\item $\mathcal{N}_x$ is a lattice in $Z(\mathfrak{g}_x)$.
\item The dimension of $Z(\mathfrak{g}_x)$ equals the rank of $\pi_2(S,x)$.
\end{enumerate}
\end{proposition}

\begin{proof} Let $\zeta= Z(\mathfrak{g}_x)$ and
denote by $\widetilde{\mathcal{N}}_x$ the image of $\partial_x$. The discreteness of $\mathcal{N}_x$ is equivalent to that of
$\widetilde{\mathcal{N}}_x$ \cite{CrFe2}, hence to the smoothness of $G_x$. The compactness of $G_x$ is equivalent to the following two
conditions: $\pi_0(G_x)$ is finite and the connected component of the identity $G_{x}^{\circ}$ is compact. From Proposition
\ref{smoothness-all-in-one}, the first condition is equivalent to $\pi_1(S)$-finite, while the second one to $(Z\times
\zeta)/\widetilde{\mathcal{N}}_x$ being compact. The last condition is equivalent to $\zeta/\mathcal{N}_x$ being compact (hence to
$\mathcal{N}_x$ being a lattice in $\zeta$). For this, note that $\zeta/\mathcal{N}_x$ injects naturally into $(Z\times
\zeta)/\widetilde{\mathcal{N}}_x$ and there is a surjection of $Z\times \zeta/\mathcal{N}_x$ onto $(Z\times \zeta)/\widetilde{\mathcal{N}}_x$.
Condition (4) is equivalent to (3) from Proposition \ref{Prop_restatement_1}.
\end{proof}

\begin{example}\label{main-example}\rm \ We now give an example in which all conditions of the theorem are satisfied, except for the vanishing of
$H^2(P_x)$, and in which the conclusion of the theorem fails. Consider the $2$-sphere $S^2$, with coordinates denoted $(u, v, w)$, endowed with
the Poisson structure $\pi_{S^2}$ which is the inverse of the area form
\[ \omega_{S^2}=  (udv\wedge dw+ vdw\wedge du+ wdu\wedge dv) .\]
Consider also the linear Poisson structure on $so(3)^*\cong \mathbb{R}^3=\{(x,y,z)\}$:
\[ \pi_{\textrm{lin}}=x\frac{\partial}{\partial y}\wedge\frac{\partial}{\partial z}+y\frac{\partial}{\partial z}\wedge\frac{\partial}{\partial x}+z\frac{\partial}{\partial x}\wedge\frac{\partial}{\partial y}.\]
Its symplectic leaves are the spheres of radius $r>0$, $S^{2}_{r}$, with the symplectic form
\[ \omega_r= \frac{1}{r^2} (x dy\wedge dz+ y dz\wedge dx+ z dx\wedge dy),\]
and the origin. Finally, let $(M, \pi_0)$ be the product of these two Poisson manifolds
\[ M= S^2\times \mathbb{R}^3,  \ \pi_0= \pi_{S^2}+ \pi_{\textrm{lin}}.\]
The symplectic leaves are: $(S,\omega_S):=(S^{2}\times \{0\}, \omega_{S^2})$ and, for $r>0$, $(S^{2}\times S^{2}_{r},
\omega_{S^2}+\omega_r)$. The abstract Atiyah sequence above $S$ is the product of the Lie algebroids $TS$ and $so(3)$. Hence, for
$x\in S$, $G_x$ equals to $G(so(3))= \textrm{Spin}(3)$ and the Poisson homotopy bundle is $P_x= S^2\times
\textrm{Spin}(3)\cong S^2\times S^3$. Using the trivial connection on $P_x$, one finds that $(M, \pi_0)$ coincides with the
resulting local model. Note that all the conditions of the theorem are satisfied, except for the vanishing of $H^2(P_x)$.

Let us now modify $\pi_0$ without modifying $j^{1}_{S}\pi_0$; we consider
\[ \pi= (1+ r^2)\pi_{S^2}+ \pi_{\textrm{lin}},\]
Note that $\pi$ has the same leaves as $\pi_0$, but with different symplectic forms:
\[(S,\omega_S), \ \ \  (S^2\times S^{2}_{r},\frac{1}{1+ r^2} \omega_{S^2}+ \omega_r), \ r>0.\]
We claim that $\pi$ is not Poisson diffeomorphic, around $S$, to $\pi_0$. Assume it is. Then, for any $r$ small enough, we find $r'$ and a
symplectomorphism
\[ \phi: (S^2\times S^{2}_{r}, \frac{1}{1+ r^2} \omega_{S^2}+ \omega_r)\rmap (S^2\times S^{2}_{r'}, \omega_{S^2}+ \omega_{r'}).\]
Comparing the symplectic volumes, we find $r'=r/(1+r^2)$.
On the other hand, $\phi$ sends the first generator $\gamma_1$ of $\pi_2(S^2\times S^{2}_{r})$ into a combination $m\gamma_1+ n\gamma_2$ with
$m$ and $n$ integers. Computing the symplectic areas of these elements, we obtain:
\[ \int_{\gamma_1} (\frac{1}{1+r^2} \omega_{S^2}+ \omega_r)= \int_{m\gamma_1+ n\gamma_2} (\omega_{S^2}+ \omega_{r'}),\]
thus $1/(1+r^2)= m+ nr'= m+ nr /(1+r^2)$. This cannot be satisfied for all $r$ (even small enough), because it forces $r$ to be an algebraic
number.

Also, the monodromy group of the leaf $S^2\times S^2_{r}$ is the subgroup of $\mathbb{R}$ generated by $4\pi$ and $4\pi(-2r)/(1+r^2)^2$; therefore, if it is discrete, then again $r$ is algebraic. This shows that $\pi$ is not integrable on any open neighborhood of $S$.
\end{example}

\begin{example}[the regular case]\label{second-main-example}\rm
Assume that $(M, \pi)$ is regular and fix a symplectic leaf $S$. Then the resulting normal form is a refinement of local Reeb
stability: the model for the underlying foliation is the same, while our theorem also specifies the leafwise symplectic forms.
To see this, let
\[ \Gamma= \pi_1(S, x),\quad \nu_x= T_xM/T_xS \]
In this case $\mathfrak{g}_x= \nu_{x}^{*}$ is abelian
and Proposition \ref{smoothness-all-in-one} gives a short exact sequence
\[ G_{x}^{\circ}\rmap G_x\rmap \Gamma .\]
Hence $G_{x}^{\circ}$ is abelian, the Poisson homotopy bundle is a principal $G_{x}^{\circ}$-bundle over the universal cover  $\widetilde{S}$;
in conclusion, as a foliated manifold, the local model is
\[ P_x\times_{G_x} \nu_x\cong \widetilde{S}\times_{\Gamma} \nu_x.\]
The Poisson structure comes from a family of symplectic forms $\omega_{t}$ on $\widetilde{S}$ of the form:
\[ \omega_{t}= p^*(\omega_S)+ t_1\omega_1+ \ldots +t_q\omega_q,\ \ t=(t_1,\ldots, t_q)\in \nu_x,\]
where we have chosen a basis of $\nu_x$, $p:\widetilde{S}\rmap S$ is the projection and $\omega_i\in \Omega^2(\widetilde{S})$ are
representatives of the components of the monodromy map $\partial: \pi_2(S)\rmap \nu_{x}^{*}$, i.e.\
\[ \partial(\sigma)=(\int_{\sigma}\omega_1,\ldots,\int_{\sigma}\omega_q),\ [\sigma]\in \pi_2(S,x).\]
Note that this condition determines uniquely $[\omega_i]\in H^2(\widetilde{S})$.
Since the rank of $\pi_2(S)$ equals the $b_2(\widetilde{S})$, by Proposition \ref{Prop_restatement_2}, the assumptions of
the theorem become: $\widetilde{S}$ is compact (so local Reeb stability applies) and $[\omega_1],\ldots, [\omega_q]$ is
a basis of $H^2(\widetilde{S})$.

Also the description of the monodromy map simplifies as shown in \cite{CrFe2}. Let $\sigma$ be a 2-sphere in $S$, with north pole
$P_N$ at $x$ and $t\in\nu_x$. Consider a smooth family $\sigma_{\epsilon}$, defined for $\epsilon$ small, of leafwise 2-spheres such that
$\sigma_0=\sigma$ and the vector $\dot{\sigma}_{0}(P_N)$ represents $t$. Then the monodromy map on $\sigma$ is:
\begin{equation}\label{monodromy-regular}
 \partial(\sigma)(t)= \frac{d}{d\epsilon} _{|_{\epsilon= 0}}\int_{\sigma_{\epsilon}} \omega_{\epsilon}.
\end{equation}

For instance, when $S$ is simply connected, then by Reeb stability $M= S\times \mathbb{R}^q$ with the trivial foliation so that the Poisson
structure on $M$ is determined by a family $\{\omega_t\}$ of symplectic forms with $\omega_0= \omega_{S}$. Then $\omega_{k}$ is simply
\[\omega_k=\frac{\partial}{\partial t_k}_{|t=0} \omega_t.\]

This class of examples reveals again the necessity of the condition (\ref{subtle-cond}). In order to see this, observe that the other conditions
of the theorem are equivalent to the first three conditions of Proposition \ref{Prop_restatement_2}, whereas the first two are automatically
satisfied in this case. By the third condition, that $\mathcal{N}_x$ is a lattice in $\nu_x^*$, we can choose a basis of $\nu_x^*$ which is also
a $\mathbb{Z}$-basis of $\mathcal{N}_x$. As a result, the $[\omega_i]$'s belong to $H^2(S;\mathbb{Z})$ and are linearly independent. If the
fourth condition is not satisfied, we can find a closed $2$-form $\lambda$, with $[\lambda]\in H^2(S;\mathbb{Z})$ and which is not in the span
of $[\omega_1], \ldots, [\omega_q]$. The Poisson structure corresponding to the family of $2$-forms
\[\omega_t:=\omega_S+t_1\omega_1+\ldots+t_q\omega_q+ t_{1}^{2} \lambda \]
satisfies all the conditions of the theorem except for (\ref{subtle-cond}), but it is not linearizable around $S$. Otherwise, we could find a
diffeomorphism of the form $(x,t)\mapsto(\phi_t(x), \tau(t))$ with $\tau(0)= 0$, $\phi_0(x)= x$ and such that
\[ \phi_{t}^{*}\omega_S+ \sum \tau_i(t)\phi_{t}^{*}\omega_i= \omega_S+ \sum t_i\omega_i+ t_{1}^{2}\lambda.\]
Since $\phi_{t}^{*}$ is the identity in cohomology, we get a contradiction.

Related to our Proposition \ref{main-cor}, we see that this Poisson structure is not integrable on any open neighborhood of $S$. This follows by
computing the monodromy groups using (\ref{monodromy-regular}) at $(t_1,\ldots,t_q)$; they are discrete if and only if $t_1\in\mathbb{Q}$.
\end{example}

\begin{example}[Duistermaat-Heckman variation formula] \rm
Next, we indicate the relationship of our theorem with the theorem of Duistermaat and Heckman 
from \cite{DH}. We first recall their result. Let $(M, \omega)$ be a symplectic manifold endowed with a Hamiltonian
action of a torus $T$ with proper moment map $J: M\rmap \mathfrak{t}^*$ and let $\xi_0\in \mathfrak{t}^*$ be a regular value of $J$. Assume
that the action of $T$ on $J^{-1}(\xi_0)$ is free. Let $U$ be a ball around $\xi_0$ consisting of regular values of $J$. The
symplectic quotients
\[ S_{\xi}:= J^{-1}(\xi)/T \]
come with symplectic forms denoted $\sigma_{\xi}$, $\xi\in U$. There are canonical isomorphisms
\[ H^2(S_{\xi})\cong H^2(S_{\xi_0})\]
for $\xi\in U$ and the Duistermaat-Heckman theorem asserts that, in cohomology,
\[ [\sigma_{\xi}]= [\sigma_{\xi_0}]+ \langle c, \xi- \xi_0\rangle, \]
where $c$ is the Chern class of the $T$-bundle $J^{-1}(\xi_0)\rmap S_{\xi_0}$. This is related to our theorem applied to the Poisson manifold
$N/ T$, where $N:=J^{-1}(U)$. The symplectic leaves of $N/T$ are precisely the $S_{\xi}$'s and $N/T$ is integrable by the symplectic groupoid
(see Proposition 4.6 in \cite{FeOrRa})
\[\mathcal{G}:=(N\times_J N)/T\rightrightarrows N/T,\]
with symplectic structure induced by $pr_1^*(\omega)-pr_2^*(\omega)\in \Omega^2(N\times N)$. The isotropy groups of $\mathcal{G}$ are all
isomorphic to $T$ and the $s$-fibers are isomorphic (as principal $T$-bundles) to the fibers of $J$. As shown in \cite{DH}, if $U$ is small
enough all fibers of $J$ are diffeomorphic as $T$-bundles, so if we are assuming that $J^{-1}(\xi_0)$ is 1-connected, then $\mathcal{G}$ is the
1-connected symplectic groupoid integrating $N/T$. In particular the Poisson homotopy bundle corresponding to $S_{\xi_0}$ is the $T$-bundle
$J^{-1}(\xi_0)\rmap S_{\xi_0}$, hence the Chern class $c$ is the same appearing in the construction of the local model (see Example
\ref{torus-bundle}) and is also the monodromy map (\ref{monodromy-regular}), interpreted as a cohomology class with coefficients in $\nu_{x}=
\mathfrak{t}$. In this case the condition $H^2(J^{-1}(\xi_0))=0$ is not required, since we can apply directly Proposition \ref{main-cor}.
\end{example}

\begin{example}\label{example-linear-Poisson}\rm
Consider $(\mathfrak{g}^*,\pi_{\mathrm{lin}})$, the dual of a Lie algebra $\mathfrak{g}$, endowed with the linear Poisson structure. Let $\xi\in
\mathfrak{g}^{*}$. As we have seen in Example \ref{P-bdle-ex1}, the leaf through $\xi$ is the coadjoint orbit $\mathcal{O}_{\xi}$
and the corresponding Poisson homotopy bundle is $P_{\xi}= G$, the 1-connected Lie group integrating $\mathfrak{g}$. So the hypothesis of our
theorem reduces to the condition that $\mathfrak{g}$ is semi-simple of compact type.
Note also that the resulting local form around $\mathcal{O}_{\xi}$ implies the linearizability of the transversal Poisson structure \cite{Wein} to $\mathcal{O}_{\xi}$, which fails for general Lie algebras \cite{Wein}-Errata. 

Of course, one may wonder about a direct argument. This is possible- and actually reveals a slightly weaker hypothesis: one needs that
$\mathcal{O}_{\xi}$ is an embedded submanifold and that $\xi\in\mathfrak{g}$ is split in the sense that there is a $G_{\xi}$-invariant
projection $p:\mathfrak{g}\rmap \mathfrak{g}_{\xi}$. For the details, we use \cite{SymplecticFibrations} and \cite{Montgomery}. Observe that the
normal bundle of $\mathcal{O}_{\xi}$ is isomorphic to $G\times_{G_{\xi}}\mathfrak{g}_{\xi}^*$. The projection $p$ induces a tubular
neighborhood:
\begin{equation*}
\varphi:G\times_{G_{\xi}}\mathfrak{g}_{\xi}^{*}\to\mathfrak{g}^{*},\quad
[g,\eta]\to Ad_{g^{-1}}^*(\xi+p^{*}(\eta)),
\end{equation*}
and a $G$-invariant principle connection on $G\rmap \mathcal{O}_{\xi}$:
\[\theta\in\Omega^1(G;\mathfrak{g}_{\xi}), \quad \theta_g=l_{g^{-1}}^*(p).\]
Let $\Omega\in \Omega^2(G\times\mathfrak{g}^*_{\xi})$ be the resulting 2-form (see subsection \ref{The local model}). The nondegeneracy locus of
$\Omega$ can be described more explicitly. Let $\mathcal{N}\subset \mathfrak{g}_{\xi}^*$ be the set of points $\eta\in \mathfrak{g}_{\xi}^*$,
for which the coadjoint orbit through $\xi+p^*(\eta)$ and the affine space $\xi+p^*(\mathfrak{g}_{\xi}^*)$ are transversal at $\xi+p^{*}(\eta)$.
section 2.3.1 and Theorem 2.3.7 in \cite{SymplecticFibrations} show that:
\begin{itemize}
\item $G\times\mathcal{N}$ is the open in $G\times\mathfrak{g}^*_{\xi}$ on which $\Omega$ is nondegenerate,
\item $G\times_{G_{\xi}}\mathcal{N}$ is the open in $G\times_{G_{\xi}}\mathfrak{g}^*_{\xi}$ on which the differential of
$\varphi$ is invertible.
\end{itemize}
Let $\pi_{\xi}$ denote the Poisson structure on $G\times_{G_{\xi}}\mathcal{N}$ obtained by reduction. Theorem 1 from section 1.3 in
\cite{Montgomery} shows that $\varphi$ is a Poisson map
\[\varphi:(G\times_{G_{\xi}}\mathcal{N},\pi_{\xi})\rmap(\mathfrak{g}^*,\pi_{\mathrm{lin}}).\]
We deduce that, around the embedded $\mathcal{O}_{\xi}$, $\varphi$ provides a Poisson diffeomorphism.
\end{example}

\section{Poisson structures around a symplectic leaf: the algebraic framework}
\label{Poisson structures around a symplectic leaf: the algebraic framework}

In this section we discuss the algebraic framework which encodes the behavior of Poisson structures around symplectic leaves (an improvement
of the framework of \cite{CrFe-stab}). This will allow us to regard the local model as a first order approximation and to
produce smooth paths of Poisson bivectors (to be used in the next section).

Since we are interested in the local behavior of Poisson structures around an embedded symplectic leaf, we may restrict our attention to a
tubular neighborhood. Throughout this section $p: E\rmap S$ is a vector bundle over a manifold $S$. Consider the vertical sub-bundle and the
space of vertical multi-vector fields on $E$ denoted
\[ V:=\ker(dp)\subset TE,\ \ \ \ \ \ \mathfrak{X}^{\bullet}_{\mathrm{V}}(E)=\Gamma(\wedge^{\bullet} V) \subset \mathfrak{X}^{\bullet}(E).\]
We have that $\mathfrak{X}^{\bullet}_{\mathrm{V}}(E)$ is a graded Lie sub-algebra of the space $\mathfrak{X}^{\bullet}(E)$ of multivector fields
on $E$, with respect to the Schouten bracket. The Lie algebra grading is:
\[ \textrm{deg}(X):= |X|-1= q- 1\ \ \ \textrm{for}\ X\in \mathfrak{X}^{q}(E).\]
Given a vector bundle $F$ over $S$, denote the space of $F$-valued forms on $S$ by:
\[ \Omega^{\bullet}(S, F):= \Gamma(\Lambda^{\bullet}T^*S\otimes F)= \Omega^{\bullet}(S)\otimes_{C^{\infty}(S)}\Gamma(F).\]
More generally, for any $C^{\infty}(S)$-module $\mathfrak{X}$, denote by
\[\Omega^{\bullet}(S, \mathfrak{X}):=\Omega^{\bullet}(S)\otimes_{C^{\infty}(S)} \mathfrak{X},\]
the space of antisymmetric forms on $S$ with values in $\mathfrak{X}$.

\subsection{The graded Lie algebra $(\widetilde{\Omega}_E,[\cdot,\cdot]_{\ltimes})$ and horizontally non-degenerate Poisson structures}
 \label{The graded Lie algebra}

We first recall the graded Lie algebra $\Omega_{E}$ of \cite{CrFe-stab}. We introduce $\Omega_{E}$ as the bi-graded vector space whose elements
of bi-degree $(p, q)$ are $p$-forms on $S$ with values in the $C^{\infty}(S)$-module $\mathfrak{X}^{q}_{\mathrm{V}}(E)$:
\[\Omega^{p,q}_{E}=\Omega^{p}(S, \mathfrak{X}^{q}_{\mathrm{V}}(E)).\]
The $\mathbb{Z}$- grading is $\textrm{deg}= p+q-1$ on $\Omega^{p,q}_{E}$, and the bracket is
\[[\varphi\otimes X,\psi\otimes
Y]=(-1)^{|\psi|(|X|-1)}\varphi\wedge\psi\otimes[X,Y].\]
We will need an enlargement $\widetilde{\Omega}_E$ of $\Omega_E$. As a bi-graded vector space, it is
\[ \widetilde{\Omega}_{E}= \Omega^{\bullet}(S, \mathfrak{X}^{\bullet}_{\mathrm{V}}(E)) + \Omega^{\bullet}(S, \mathfrak{X}_{\mathrm{P}}(E)) \subset \Omega^{\bullet}(S, \mathfrak{X}^{\bullet}(E)), \]
where $\mathfrak{X}_{\mathrm{P}}(E)$ is the space of \textbf{projectable vector fields} on $E$, i.e. vector fields $X\in \mathfrak{X}(E)$ with
the property that there is a vector field on $S$, denoted $p_S(X)\in \mathfrak{X}(S)$, such that $dp(X)=p_S(X)$. Hence, in bi-degree $(p, q)$ we
have
\begin{eqnarray*}
\nonumber \widetilde{\Omega}_{E}^{p,q}= \left\{
\begin{array}{rl}
\Omega^{p}(S, \mathfrak{X}^{q}_{\mathrm{V}}(E)) & \text{if } q\neq 1\\
\Omega^{p}(S, \mathfrak{X}_{\mathrm{P}}(E)) & \text{if } q = 1\\
\end{array} \right..
\end{eqnarray*}

The relationship between $\widetilde{\Omega}_E$ and $\Omega_E$ is similar to the one between the Lie algebras $\mathfrak{X}_{V}(E)$ and
$\mathfrak{X}_{\mathrm{P}}(E)$: we have a short exact sequence of vector spaces:
\[ 0\rmap \Omega_E\rmap \widetilde{\Omega}_E\stackrel{p_S}{\rmap} \Omega(S, TS) \rmap 0.\]
Next, we show that this is naturally a sequence of graded Lie algebras. On $\Omega(S, TS)$ we consider the
\textbf{Fr\"{o}hlicher-Nijenhuis-Bracket}, denoted $[\cdot, \cdot]_{F}$, which we recall using section 13 of \cite{Kumpera-Spencer}. The
key-point is that $\Omega(S, TS)$ can be identified with the space of derivations of the graded algebra $\Omega(S)$, which commute with the
DeRham differential, and, as a space of derivations, it inherits a natural Lie bracket. In more detail, for $u= \alpha\otimes X\in \Omega^r(S,
TS)$, the operator $\mathcal{L}_u:= [i_u, d]$ on $\Omega(S)$ is given by:
\[ \mathcal{L}_{u}(\omega)= \alpha\wedge \mathcal{L}_{X}(\omega)+ (-1)^r d\alpha \wedge i_X(\omega) .\]
The resulting commutator bracket on $\Omega(S, TS)$ is:
\[ [u, v]_{F}= \mathcal{L}_{u}(\beta)\otimes
Y-(-1)^{rs}\mathcal{L}_{v}(\alpha)\otimes X+ \alpha\wedge\beta\otimes[X,Y]\] for $u= \alpha\otimes X\in \Omega^r(S, TS)$, $v= \beta\otimes Y\in
\Omega^s(S, TS)$. With these $(\Omega(S, TS),[\cdot,\cdot]_F)$ is a graded Lie algebra, where the grading $\textrm{deg}= r$ on $\Omega^r(S, TS)$.
Consider the element corresponding to the identity map of $TS$, denoted by:
\[ \gamma_S\in \Omega^1(S, TS)\]
Then $\gamma_S$ is central in $(\Omega(S, TS),[\cdot,\cdot]_{F})$ and it represents the DeRham differential:
\begin{equation} \label{gamma-represents-d}
\mathcal{L}_{\gamma_S}=d:\Omega^{\bullet}(S)\rmap \Omega^{\bullet +1}(S).
\end{equation}

Next, the operations involving $\Omega(S, TS)$ have the following lifts to $E$:
\begin{itemize}
\item With the short exact sequence
\[ 0\rmap \Omega(S, \mathfrak{X}_{\mathrm{V}}(E))\rmap \Omega(S, \mathfrak{X}_{\mathrm{P}}(E))\stackrel{p_S}{\rmap} \Omega(S, TS)\rmap 0\]
in mind, there is a natural lift of $[\cdot, \cdot]_{F}$ to the middle term,
which we denote by the same symbol. Actually, realizing
\[\Omega(S, \mathfrak{X}_{\mathrm{P}}(E))\stackrel{p^*}{\hookrightarrow} \Omega(E, TE),\]
this is just the restriction of the Fr\"{o}hlicher-Nijenhuis-Bracket on $\Omega(E, TE)$.
\item The action $\mathcal{L}$ of $\Omega(S, TS)$ on $\Omega(S)$ induces an action of
$\Omega(S, \mathfrak{X}_{\mathrm{P}}(E))$ on $\Omega_E$, for
$u=\alpha\otimes X\in \Omega(S, \mathfrak{X}_{\mathrm{P}}(E))$ and
$v=\omega\otimes Y\in \Omega_E$, we have:
\[ \mathcal{L}_{u}(v)= \mathcal{L}_{p_S(u)}(\omega)\otimes Y+ \alpha\wedge \omega\otimes [X, Y].\]
\end{itemize}

Putting everything together, the following is straightforward:

\begin{proposition}
The following bracket defines a graded Lie algebra on $\widetilde{\Omega}_{E}$:
\[[u,v]_{\ltimes}=\left\{
\begin{array}{cc}
[u,v]&\textrm{for } u, v\in \Omega_E, \\
\mathcal{L}_{u}(v) &
\textrm{for } u\in \Omega(S, \mathfrak{X}_{\mathrm{P}}(E)), v\in\Omega_E,\\
\phantom{} [u,v]_{F} &
\textrm{for } u,v\in \Omega(S, \mathfrak{X}_{\mathrm{P}}(E)),
\end{array}
 \right.
\]
Moreover, we have a short exact sequence of graded Lie algebras:
\[0\to(\Omega_{E}^{\bullet},[\cdot,\cdot])\to (\widetilde{\Omega}^{\bullet}_E,[\cdot,\cdot]_{\ltimes})\stackrel{p_S}{\to} (\Omega^{\bullet}(S, TS),[\cdot,\cdot]_F)\to 0.\]
\end{proposition}

We encode a bit more of the structure of the algebra $(\widetilde{\Omega}_E,[\cdot,\cdot]_{\ltimes})$ in the following lemma, whose proof is
also straightforward:

\begin{lemma}\label{OmegaS-Central}
Identifying $\Omega(S)\cong p^*(\Omega(S))\subset \widetilde{\Omega}_E$, $\Omega(S)$ is a central ideal in $\Omega_E$. The induced
representation of $\widetilde{\Omega}_E$ on $\Omega(S)$ factors through $p_S$, i.e.\
\[[u,\omega]_{\ltimes}=\mathcal{L}_u(\omega)=\mathcal{L}_{p_S(u)}(\omega),\  (\forall)\ u\in \widetilde{\Omega}_E,\  \omega\in\Omega(S).\]
\end{lemma}

As an illustration of the use of $\widetilde{\Omega}_E$, we look at \textbf{Ehresmann connections} on $E$. Viewing such a connection as a
$C^{\infty}(S)$-linear map which associates to a vector field $X$ on $S$ its horizontal lift to $E$, we see that it is the same thing as an
element in $\Gamma\in \widetilde{\Omega}_{E}^{1, 1}$ which satisfies $p_S(\Gamma)= \gamma_S$. Also the curvature $R_{\Gamma}$ of $\Gamma$ is
just $R_{\Gamma}=\frac{1}{2}[\Gamma,\Gamma]_{\ltimes}\in\Omega^{2,1}_E$.

We introduce the following generalization of flat Ehresmann connections (the terminology will be explained in Remark \ref{lomod-Dirac} below).

\begin{definition} A \textbf{Dirac element} on $p:E\rmap S$ is an element $\gamma\in \widetilde{\Omega}_{E}^{2}$, satisfying
\[ [\gamma, \gamma]_{\ltimes}= 0, \ \ \ \ pr_{S}(\gamma)= \gamma_S.\]
We use the following notations for the components of $\gamma$:
\begin{itemize}
\item $\gamma^{\mathrm{v}}$ for the $(2, 0)$ component- an element in $\mathfrak{X}_{\mathrm{V}}^{2}(E)$.
\item $\Gamma_{\gamma}$ for the $(1, 1)$ component- an Ehresmann connection on $E$.
\item $\mathbb{F}_{\gamma}$ for the $(0, 2)$ component- an element in $\Omega^2(S, C^{\infty}(E))= \Gamma(p^*\Lambda^2T^*S)$.
\end{itemize}
The \textbf{Poisson support} of $\gamma$ is the set of points $e\in E$ at which $\mathbb{F}_{\gamma}$ is non-degenerate.
\end{definition}

The relevance of such elements to the study of Poisson structures around a symplectic leaf comes from the fact that, while $E$ plays the role of
small tubular neighborhoods, on such $E$'s the following condition will be satisfied.


\begin{definition} A bivector $\theta\in \mathfrak{X}^2(E)$ is called \textbf{horizontally non-degenerate}
if
\[ V_e+ \theta^{\sharp}(V_{e}^{\circ})= T_eE, \ (\forall) \ e\in E,\]
 where $V_{e}^{\circ}\subset T^{*}_{e}E$ is the annihilator of $V_e$ in $T^*_eE$.
\end{definition}

Moreover, Vorobjev's Theorem 2.1 in \cite{Vorobjev} can be summarized in the following:

\begin{proposition}\label{Dirac-Poisson} There is a 1-1 correspondence between
\begin{enumerate}
\item Dirac elements $\gamma\in \widetilde{\Omega}_{E}^2$ with support equal to $E$.
\item Horizontally non-degenerate Poisson structures $\theta$ on $E$.
\end{enumerate}
\end{proposition}

The explicit construction of the 1-1 correspondence is important as well; we recall it below. The main
point of our proposition is that the list of complicated equations from \cite{Vorobjev} takes now the compact form $[\gamma,
\gamma]_{\ltimes}= 0$. So, let $\theta$ be a horizontally non-degenerate Poisson structure on $E$. The non-degeneracy  implies that
\[ H_{\theta}= \theta^{\sharp}(V^{\circ})\]
gives an Ehresmann connection on $E$; we denote it by $\Gamma_{\theta}\in \widetilde{\Omega}_{E}^{1, 1}$.
 With respect to the resulting decomposition of $TE$, the mixed component of $\theta$ vanishes, i.e.
\[ \theta= \theta^{\mathrm{v}}+ \theta^{\textrm{h}}\in \Lambda^2V\oplus \Lambda^2H_{\theta}.\]
The first term is the desired $(2, 0)$-component. The second one is non-degenerate, thus, after passing from $H_{\theta}$ to $TS$ and then
taking the inverse, we get an element
\[ \mathbb{F}_{\theta}\in  \Gamma(p^*\Lambda^2T^*S).\]
This will be the desired $(0, 2)$-component. Explicitly,
\begin{equation}\label{2Form}
\mathbb{F}_{\theta}(dp_e(\theta^{\textrm{h} \sharp}\eta),dp_e(\theta^{\textrm{h}
\sharp}\mu))=-\theta^{\textrm{h}}(\eta,\mu),\ \ (\forall)  \ \eta,\mu\in p^*(T^*S)_e.
\end{equation}
Altogether, the 1-1 correspondence associates to $\theta$ the element
\[ \gamma= \theta^{\mathrm{v}}+ \Gamma_{\theta}+ \mathbb{F}_{\theta} \in \widetilde{\Omega}^{2}_{E}\]
with $\mathbb{F}_{\theta}$ non-degenerate at all points of $E$ and $p_S(\gamma)= p_S(\Gamma_{\theta})= \gamma_{S}$.

\begin{proof} (of the proposition) Conversely, it is clear that we can reconstruct $\theta$ from $\gamma$.
One still has to check that $[\theta, \theta]= 0$ is equivalent to the equation $[\gamma, \gamma]_{\ltimes}= 0$. By Vorobjev's formulas (Theorem
2.1 in \cite{Vorobjev}) and their interpretation using $\Omega_E$ from \cite{CrFe-stab} (Theorem 4.2), the Poisson equation is equivalent to:
\[ [ \theta^{\mathrm{v}},\theta^{\mathrm{v}}]=0,\ [ \Gamma_{\theta},\theta^{\mathrm{v}}]_{\ltimes}=0,\ R_{\Gamma_{\theta}}+[ \theta^{\mathrm{v}}, \mathbb{F}_{\theta}]=0,\ [ \Gamma_{\theta},\mathbb{F}_{\theta}]_{\ltimes}=0 .\]
Here we have used the remark that the covariant exterior derivative (denoted in \cite{CrFe-stab} by $d_{\Gamma_{\theta}}$,
$\partial_{\Gamma_{\theta}}$ respectively) can be given by $ad_{\Gamma_{\theta}}=[\Gamma_{\theta},\cdot]_{\ltimes}:\Omega_{E}\to\Omega_E$. Finally,
\[0=[\gamma,\gamma]_{\ltimes}=([\theta^{\mathrm{v}},\theta^{\mathrm{v}}])+2([\Gamma_{\theta},\theta^\textrm{v}]_{\ltimes})+2(R_{\Gamma_{\theta}}+[\theta^\textrm{v},\mathbb{F}_{\theta}])+2([\Gamma_{\theta},\mathbb{F}_{\theta}]_{\ltimes})\in\]
\[\in\Omega^{0,3}_E\oplus\Omega^{1,2}_E\oplus\Omega^{2,1}_E\oplus\Omega^{3,0}_E=\Omega^{3}_E.\]
\end{proof}

\begin{remark}\label{lomod-Dirac}\rm \
Generalizing the case of Poisson structures, recall \cite{FernandesBrahic, Wade} that a Dirac structure $L\subset TE\oplus T^*E$ is called
horizontally non-degenerate if $L\cap (V\oplus V^{\circ})= \{0\}$.
The previous discussion applies with minor changes to such structures. The integrability condition (the four equations above) have been extended
to horizontally non-degenerate Dirac structures (Corollary 2.8 in \cite{FernandesBrahic} and Theorem 2.9 in \cite{Wade}). We find out that there
is a 1-1 correspondence between
\begin{itemize}
\item Dirac elements $\gamma\in \widetilde{\Omega}_{E}^{2}$.
\item Horizontally non-degenerate Dirac structures on $E$.
\end{itemize}
Moreover, in this correspondence, the support of $L$ (cf. Remark \ref{remark-Dirac}) coincides with the Poisson support of $\gamma$. Explicitly,
the Dirac structure corresponding to $\gamma$ is
\[ L_{\gamma}= \textrm{Graph}(\gamma^{\textrm{v}\sharp}:H^{\circ}\to V)\oplus \textrm{Graph}(\mathbb{F}_{\gamma}^{\sharp}:H\to V^{\circ}),\]
where we use the decomposition $TE= V\oplus H$ induced by the connection $\Gamma_{\gamma}$.
\end{remark}

Finally, we identify the Poisson cohomology complex $(\mathfrak{X}^{\bullet}(E),d_{\theta})$ ($d_{\theta}= [\theta,\cdot]$).

\begin{proposition}\label{PoissonCohomology} Let $\theta$ be a horizontally non-degenerate Poisson structure on $E$ with corresponding Dirac element $\gamma$. Then there is an isomorphism of complexes
\[ \tau_{\theta}: (\mathfrak{X}^{\bullet}(E),d_{\theta})\rmap (\Omega_{E}^{\bullet},ad_{\gamma}), \ \ \textrm{where}\  ad_{\gamma}= [\gamma, \cdot]_{\ltimes}.\]
\end{proposition}

Again, this is a reformulation of a result of \cite{CrFe-stab}, namely of Proposition 4.3, with the remark that the operator $d_{\theta}$ in
\emph{loc.cit.} is simply our $ad_{\gamma}$. For later use, we also give the explicit description of $\tau_{\theta}$. Identifying
$\Omega_E=\Gamma(\wedge(p^{*}T^*S\oplus V))$,
\begin{equation}\label{tau-explicit}
\tau_{\theta}=\wedge^{\bullet} f_{\theta,*}: \mathfrak{X}^{\bullet}(E)\rmap \Omega^{\bullet}_E,
\end{equation}
where $f_{\theta}$ is the bundle isomorphism
\[f_{\theta}:=(-\mathbb{F}_{\theta}^{\sharp},id_V):H_{\theta}\oplus V=TE\rmap p^{*}T^*S\oplus V.\]

\subsection{The dilatation operators and jets along $S$}
\label{The dilatation operators and jets along $S$}

For $t\in\mathbb{R}$, $t\neq0$ let $m_t:E\rmap E$ be the fiberwise multiplication by $t$. Pull-back by $m_t$ induces an automorphism
\[m_t^*:(\widetilde{\Omega}_E,[\cdot,\cdot])\rmap (\widetilde{\Omega}_E,[\cdot,\cdot]).\]
It preserves $\Omega_E$ and acts as the identity on $\Omega(S)$. Define the \textbf{dilation operators}: 
\[ \varphi_t:\widetilde{\Omega}_E\rmap \widetilde{\Omega}_E, \ \varphi_t(u)=t^{q-1}m_t^{*}(u),\textrm{ for }u\in \widetilde{\Omega}_E^{(\bullet,q)}.\]

\begin{remark}\label{phi-local-coordinates}\rm
It is useful to describe this operation in local coordinates. Choose $(x^i)$ coordinates for $S$ and $(y^a)$ linear coordinates on the fibers of
$E$. An arbitrary element in $\Omega^p(S,\mathfrak{X}^q(E))$ is a sum of elements of type
\[ a(x, y) dx^I \otimes \partial_{x^J}\wedge \partial_{y^K} \]
where $I$, $J$ and $K$ are multi-indices with $|I|= p$, $|J|+ |K|= q$ and $a=a(x, y)$ is a smooth function. Such an element is in $\Omega_{E}$
if and only if it only contains terms with $|J|= 0$. The elements in $\widetilde{\Omega}_E$ are also allowed to contain terms with $|J|= 1$, but
those terms must have $|K|=0$, and the coefficient $a$ only depending on $x$.  Applying $\varphi_t$ to such an element we find
\[ t^{|J|-1} a(x, ty) dx^I \otimes \partial_{x^J}\wedge \partial_{y^K}.\]
\end{remark}

\begin{lemma} $\varphi_t$ preserves the bi-degree, is an automorphism of the graded Lie algebra $\widetilde{\Omega}_E$ and preserves $\Omega_E$.
\end{lemma}
\begin{proof} Due to its functoriality, $m_{t}^{*}$ has similar properties. Due to the shift degree by $1$ in the Lie degree, also the multiplication
by $t^{q-1}$ has the same properties. Hence also the composition of the two operations, i.e. $\varphi_t$, has the desired properties.
\end{proof}

Together with $\varphi_t$ we also introduce the following subspaces of $\widetilde{\Omega}_E$ for $l\in \mathbb{Z}$:
\[ \textrm{gr}_{l}(\widetilde{\Omega}_E)= \{ u\in \widetilde{\Omega}_E: \varphi_t(u)= t^{l-1} u\}\subset \widetilde{\Omega}_E,\]
\[ J^{l}_{S}(\widetilde{\Omega}_E)= \textrm{gr}_{0}(\widetilde{\Omega}_E)\oplus \ldots \oplus \textrm{gr}_{l}(\widetilde{\Omega}_E)\subset \widetilde{\Omega}_E.\]
These spaces vanish for $l< 0$. The elements in $\textrm{gr}_{0}$ are called \textbf{constant}, those in $\textrm{gr}_{1}$ are called
\textbf{linear}, while those in $\textrm{gr}_{l}$ are called \textbf{homogeneous of degree $l$}. Similarly one defines 
$\textrm{gr}_{l}(\Omega_E)$; one has (e.g. using local formulas- see Remark \ref{phi-local-coordinates})
\begin{equation}\label{computation-grading}
\textrm{gr}_{l}(\Omega_{E}^{p, q})= \Omega^p(S, \Lambda^qE\otimes S^lE^*),
\end{equation}
where we regard the sections of $E$ as fiberwise constant vertical vector fields on $E$ and those of $S^lE^*$ as degree $l$ homogeneous
polynomial functions on $E$. Moreover, $\textrm{gr}_{l}(\widetilde{\Omega}_{E}^{p, q})$ coincides with $\textrm{gr}_{l}(\Omega_{E}^{p, q})$
except for the case $l= 1$, $q= 1$ when
\[\textrm{gr}_{1}(\widetilde{\Omega}_{E}^{p, 1})= \Omega^p(S, \mathfrak{X}_{\textrm{lin}}(E)),\]
where $\mathfrak{X}_{\textrm{lin}}(E)$ is the space of linear vector fields on $E$, i.e. projectable vector fields whose flow is fiberwise
linear.

Our next aim is to introduce the partial derivative operators along $S$,
\[ d_{S}^{l}: \widetilde{\Omega}_E\rmap \textrm{gr}_l(\widetilde{\Omega}_E),\]
used to introduce the jet operators $j^{n}_{S}$. 
To define and handle them, we use the formal power series expansion of $t\varphi_{t}(u)$ with respect to $t$. Although $\varphi_t$ is not
defined at $t= 0$, it is clear (use again Remark \ref{phi-local-coordinates}) that, for any $u\in \widetilde{\Omega}_E$, the map
\[ \mathbb{R}^{*}\ni t\mapsto t\varphi_t(u)\in \widetilde{\Omega}_E \]
admits a smooth prolongation to $\mathbb{R}$. Hence the following definition makes sense.

\begin{definition} For $u\in \widetilde{\Omega}_E$ define the \textbf{$n$-th order derivatives} of $u$ along $S$,
denoted by $d^{n}_{S}u$, as the coefficients of the formal power expansion around $t= 0$:
\[ \varphi_{t}(u)\cong t^{-1} u|_{S}+ d_{S}u+ t d^{2}_{S} u+ \ldots .\]
In other words,
\[ d_{S}^{n}(u)= \frac{1}{n!}\frac{d^n}{dt^n}|_{t= 0} t\varphi_t(u) \in \widetilde{\Omega}_E.\]
For $n=0$ we also use the notation $u|_{S}$. Define the \textbf{$n$-th order jet} of $u$ along $S$ as
\[ j^{n}_{S}(u)= \sum_{k= 0}^{n} d_{S}^{k} (u).\]
\end{definition}

\begin{lemma}\label{TheFormula}
We have that $d_{S}^{n}(u)\in \textrm{gr}_{n}(\widetilde{\Omega}_E)$ and $j^{n}_{S}(u)\in J^{n}_{S}(\widetilde{\Omega}_E)$.
\end{lemma}

\begin{proof}
Since $\varphi_{r}\circ \varphi_{s}=\varphi_{rs}$, we have that $\varphi_{r}(s\varphi_{s}(u))=r^{-1}[\xi\varphi_{\xi}(u)]_{|\xi=rs}$. Taking the
$n$-th derivative at $s=0$, we obtain the first part, which implies the second.
\end{proof}

The power series description, together with the properties of $\varphi_t$, are very useful in avoiding computations. For instance, using that
$\varphi_t$ preserves $[\cdot, \cdot]_{\ltimes}$ we obtain:

\begin{lemma}\label{Newton-formula}
For any $u, v\in \widetilde{\Omega}_E$, $d_{S}^{l}[u, v]_{\ltimes}= \sum_{p+q= l+1} [d^{p}_{S}u, d^{q}_{S}v]_{\ltimes}$.
\end{lemma}

As an illustration of our constructions let us look again at connections. We have already seen that an Ehresmann connection on $E$ can be seen
as an element $\Gamma\in \widetilde{\Omega}_{E}^{1, 1}$. Hence we can talk about its restriction to $S$ as an element
\[ \Gamma|_{S}\in \Omega^1(S, E) .\]
This construction was introduced in \cite{CrFe-stab} in a more ad-hoc fashion. Also, $\Gamma$ is
linear as an element of $\widetilde{\Omega}_{E}$ if and only if it is a linear connection. For the direct implication: the properties of
$\varphi_t$ immediately imply that the $\ltimes$- bracket with $\Gamma$ preserves $\textrm{gr}_{0}(\tilde{\Omega}^{\bullet, 1}_{E})=
\Omega^{\bullet}(S, E)$ hence it induces a covariant derivative $d_{\Gamma}:= [\Gamma, \cdot ]_{\ltimes}: \Omega^{\bullet}(S, E)\rmap \Omega^{\bullet+ 1}(S, E)$.\\

From now on we will restrict our attention to Poisson structures which admit $S$
as a symplectic leaf. The following is immediate (see also Proposition 5.1 of \cite{CrFe-stab}).

\begin{lemma} \label{lma-last} Let $\pi$ be a horizontally non-degenerate Poisson structure on $E$ with corresponding Dirac element $\gamma\in \widetilde{\Omega}_{E}^{2}$
(cf. Proposition \ref{Dirac-Poisson}). Then $S$ is a symplectic leaf if and only if $\gamma|_{S}$ lives in bi-degree $(2, 0)$. In this case the
symplectic form is
\[ \omega_S:=-\gamma|_{S}\in \textrm{gr}_{0}(\widetilde{\Omega}_{E}^{2, 0})= \Omega^2(S).\]
\end{lemma}

For the first order approximation along $S$, we have:

\begin{proposition}\label{prop-first-jet} Let $\pi$ be a horizontally non-degenerate Poisson structure on $E$ which admits $S$ as
symplectic leaf, with corresponding Dirac element $\gamma\in \widetilde{\Omega}_{E}^{2}$. Then
\[ j^{1}_{S}\gamma\in J^{1}_{S}(\widetilde{\Omega}_{E}^{2})\subset \widetilde{\Omega}_{E}^{2}\]
is a Dirac element whose Poisson support $N$ is an open neighborhood of $S$ in $E$. In particular, on $N$,
it is associated with a Poisson structure, denoted
\[ j^{1}_{S}\pi\in \mathfrak{X}^2(N).\]
\end{proposition}

\begin{proof} The non-trivial part of the proposition (and which uses the fact that $S$ is a symplectic leaf) is to show that
$[j^{1}_{S}\gamma, j^{1}_{S}\gamma]_{\ltimes}= 0$. This follows by applying the similar equation for $\gamma$, using the Newton formula of Lemma
\ref{Newton-formula} to compute its first order consequences and then using the fact that $\Omega(S)$ is in the center of $\Omega_E$ (Lemma
\ref{OmegaS-Central}) to delete the term $[\gamma|_{S}, d_{S}^{2}\gamma]_{\ltimes}$.
\end{proof}

\begin{definition} The Poisson bivector $j^{1}_{S}\pi$ from the previous proposition is called \textbf{the first order approximation of $\pi$ along $S$}.
\end{definition}

\begin{remark}\label{reconciliation} We explain how this definition is compatible with Proposition \ref{1st-jet-Atiyah}
(which says that the first order information of $\pi$ along $S$ is encoded in the Atiyah sequence of $A_S$) and with Definition
\ref{first-order-jet-def1} (the local model). Let us fix the symplectic structure $\omega_S$ of the leaf. Since the entire discussion depends
only on $j^{1}_{S}\pi$, we may assume that $\pi= j^{1}_{S}\pi$, i.e. the corresponding Dirac element is the sum of
\[ \pi^{\mathrm{v}}\in \textrm{gr}_1(\widetilde{\Omega}_{E}^{0, 2})= \Gamma (\Lambda^2E\otimes E^*),\]
\[ \Gamma= \Gamma_{\pi}\in  \textrm{gr}_1(\widetilde{\Omega}_{E}^{1, 1})= \Omega^1(S, \mathfrak{X}_{\textrm{lin}}(E)),\]
\[ -\omega_{S}+\sigma= \mathbb{F}_{\pi}\in \textrm{gr}_0(\widetilde{\Omega}_{E}^{2, 0})\oplus\textrm{gr}_1(\widetilde{\Omega}_{E}^{2, 0})=\Omega^2(S)\oplus\Omega^2(S, E^*).\]
On the other hand, the algebroid structure is defined on $A_S= TS\oplus E^*$,
with anchor the first projection. What is needed in order to describe such a Lie algebroid structure are precisely the elements
$\pi^{\mathrm{v}}$, $\Gamma$ and $\sigma$. Explicitly, identifying
\[ \Gamma(E^*)= \textrm{gr}_1(\Omega_{E}^{0, 0}),\]
one obtains the following formulas for the bracket $[\cdot,\cdot]_{A}$ on $A_{S}$:
\[ [\alpha, \beta]_{A}= [\beta, [\pi^{\mathrm{v}},
\alpha]_{\ltimes}]_{\ltimes}, \ [\alpha, X]_{A}= [\Gamma, \alpha]_{\ltimes}(X), \ [X, Y]_{A}= [X, Y]+ \sigma(X, Y),\] where
$\alpha,\beta\in\Gamma(E^*)$, $X,Y\in \mathfrak{X}(S)$. This describes the 1-1 correspondence between first jets of $\pi$'s, with $(S,\omega_S)$
as a symplectic leaf, and Lie algebroid structures on $A_{S}$ with anchor the first projection- as the explicit version of our Proposition
\ref{1st-jet-Atiyah} and a more compact description of Theorem 4.1 in \cite{Vorobjev}.

To explain the compatibility with Definition \ref{first-order-jet-def1}, we first use Remark \ref{non-smooth-model} which described the local
model as a Dirac structure on the dual $K^*$ of the kernel of $A_S$, using a splitting $\theta: A_S\rmap K$. In our case, $K= E^*$ (hence the
Dirac structure is on $E$), one can use as splitting  the canonical projection and, since  all the objects involved are given by explicit
formulas, it is straightforward to compute and compare the resulting Dirac structure (coming from Remark \ref{non-smooth-model}) and the Dirac
structure corresponding to $j^{1}_{S}\pi$: the two coincide!
\end{remark}

The following proposition is important for the Moser path method. We connect $\pi$ to its first order approximation $j^{1}_{S}\pi$ by a smooth
path of Poisson structures.

\begin{proposition}\label{the-Moser-path} Let $\pi$ be a horizontally non-degenerate Poisson structure on $E$ which admits $S$ as a symplectic leaf. Let $\gamma\in \widetilde{\Omega}_{E}^{2}$ be the corresponding Dirac element. Then, on a small enough neighborhood $N$ of $S$, there exists a smooth path of Poisson structures
\[ \pi_{t}\in \mathfrak{X}^2(N),\ \ \ \textrm{with}\ \ \pi_{1}= \pi|_{N}, \ \pi_{0}= j^{1}_{S}(\pi).\]
More precisely, $\pi_t$ can be chosen via the smooth path of Dirac elements
\[ \gamma_t= \gamma|_{S}+ \frac{t\varphi_t(\gamma)- \gamma|_{S}}{t} \ \ \ \forall \ t\in (0, 1],\]
and $N$ is the intersection of the Poisson supports of these elements.
\end{proposition}


\begin{proof} Note first that $\gamma_t$ extends smoothly at $t= 0$ as $\gamma_0= j^{1}_{S}\gamma$. We will perform the computations for $t\neq 0$ and extended to $t= 0$ by continuity. We write:
\[ \gamma_t= \varphi_t(\gamma)+ (t^{-1}-1)\omega_{S}= \varphi_t(\gamma+ (1-t)\omega_{S}).\]
Using that $\varphi_t$ commutes with the bracket, $\gamma$ is a Dirac element and $[\omega_S,\omega_S]_{\ltimes}=0$:
\[ [\gamma_t, \gamma_t]_{\ltimes}= \varphi_t([\gamma+(1-t)\omega_S,\gamma+(1-t)\omega_S]_{\ltimes})=2(1-t)\varphi_t([\gamma,\omega_S]_{\ltimes}).\]
By Lemma \ref{OmegaS-Central}, formula (\ref{gamma-represents-d}), and the fact that $\omega_S$ is symplectic we have that
\[[\gamma,\omega_S]_{\ltimes}= \mathcal{L}_{p_S(\gamma)}(\omega_S)=\mathcal{L}_{\gamma_S}(\omega_S)=d\omega_S=0.\]
This proves that $[\gamma_t, \gamma_t]_{\ltimes}= 0$. It is also clear that $p_{S}(\gamma_t)= p_S(\varphi_t(\Gamma_{\pi}))= p_{S}(\Gamma_{\pi})= \gamma_{S}$.
Finally, note that $\gamma_{t|S}=-\omega_S$, hence  $(S,\omega_S)$ is a symplectic leaf for all $\gamma_t$'s. 
This also show that $S\subset N$. To see that $N$ is indeed open, we apply the topological tube lemma:
from the smoothness of the family $\gamma_t$,  $\mathcal{N}$ consisting of pairs $(t, e)$ with $e$ in the support of $\gamma_t$ is open
in $[0, 1]\times E$; also, $e\in N$ means $[0, 1]\times \{e\}\subset \mathcal{N}$ hence, by the tube lemma, $[0, 1]\times W\subset \mathcal{N}$
(hence $W\subset N$) for some open $W\ni e$.
\end{proof}

\section{Proof of the main theorem; step 1: Moser path method}
\label{Proof of the main theorem; step 1: Moser path method}

In this section we use the Moser path method to reduce the proof of the main theorem to some cohomological equations. The main outcome is
Theorem \ref{theorem-1} below.

Let $(M, \pi)$ is a Poisson manifold and let $(S,\omega_S)$ be a symplectic leaf. We start by describing the relevant cohomologies.
They are all relatives of the Poisson cohomology groups $H^{\bullet}_{\pi}(M)$- defined by the complex $(\mathfrak{X}^{\bullet}(M), d_{\pi})$
where $d_{\pi}= [\pi, \cdot]$. The first one is, intuitively, the Poisson cohomology of the germ of $(M, \pi)$ around $S$:
\[ H^{\bullet}_{\pi}(M)_{S}= \lim_{S\subset U} H^{\bullet}_{\pi|_{U}}(U),\]
where the limit is the direct limit over all the open neighborhoods $U$ of $S$ in $M$. The next relevant cohomology, the \textbf{Poisson
cohomology restricted to $S$}, denoted
\[H^{\bullet}_{\pi, S}(M),\]
is defined by the complex $(\mathfrak{X}^{\bullet}_{|S}(M), d_{\pi}|_{S})$, where $\mathfrak{X}_{|S}(M)= \Gamma(\Lambda^{\bullet} TM|_{S})$.

The last relevant cohomology is a version of $H^{\bullet}_{\pi, S}(M)$ with coefficients in the conormal bundle of $S$, which is best described using Lie algebroids.
Given a Lie algebroid $A$ and a representation $(V,\nabla)$ of $A$, the cohomology of $A$ with coefficients in $V$, denoted $H^{\bullet}(A, V)$,
is defined by the complex $\Omega^{\bullet}(A, V)= \Gamma(\Lambda^{\bullet}A^*\otimes V)$ endowed with the differential $d_{A, \nabla}$
given by the classical Koszul formula (see \cite{MackenzieLL}). When $V$ is the trivial 1-dimensional representation, one obtains the cohomology
DeRham cohomology of $A$. For instance, $H^{\bullet}_{\pi}(M)$ is just the cohomology of the cotangent algebroid $T^*M$, while $H^{\bullet}_{\pi, S}(M)$ is just the cohomology of the restriction
\[ A_{S}:= T^{*}M|_{S}.\]
The conormal bundle $\nu_{S}^{*}$ of $S$ is a representation of $A_{S}$ with $\nabla_{\alpha}(\beta)= [\alpha, \beta]$ (where $[\cdot, \cdot]$ is the bracket of $A_S$).
The last relevant cohomology is
\[ H^{\bullet}_{\pi, S}(M, \nu_{S}^{*}):= H^{\bullet}(A_S, \nu_{S}^{*}).\]

\begin{theorem}\label{theorem-1}
 Let $S$ be an embedded symplectic leaf of a Poisson manifold $(M, \pi)$ and let $j^{1}_{S}\pi$ be the
first order approximation of $\pi$ along $S$ associated to some tubular neighborhood of $S$ in $M$. If
\[  H^{2}_{\pi}(M)_{S}= 0, \ H^{1}_{\pi, S}(M)= 0, \ H^{1}_{\pi, S}(M, \nu_{S}^{*})= 0,\]
then, around $S$, $\pi$ and $j^{1}_{S}\pi$ are Poisson diffeomorphic, by a Poisson diffeomorphism which is the identity on $S$.
\end{theorem}

The rest of this section is devoted to the proof of the theorem, followed by a slight improvement that will be
used in order to prove Proposition \ref{main-cor}.

First of all, by using a tubular neighborhood, we may assume that $M= E$ is a vector bundle $E$ over $S$ and $\pi$ is horizontally
non-degenerate. Let $\gamma\in \widetilde{\Omega}_{E}^{2}$ be the associated Dirac element. We first rewrite in terms of $E$ and $\gamma$ the
complexes computing the last two cohomologies in the statement of the theorem. 

\begin{lemma}\label{cohom-restricted-algebroid} For any $l\geq 0$, the complex $(\Omega^{\bullet}(A_S, S^lE^*), d_A)$ computing the
cohomology of $A_S$ with coefficients in the $l$-th symmetric power of $\nu_{S}^{*}= E^*$ is canonically
isomorphic to the complex $(\textrm{gr}_{l}(\Omega_{E}^{\bullet}), [d^{1}_{S}\gamma, \cdot])$.
\end{lemma}

\begin{proof}
We use the notations and the explicit formulas from Remark \ref{reconciliation}. With the identification $A_S= TS\oplus E^*$ and the
identification (\ref{computation-grading}), we see that
\[ \textrm{gr}_{l}(\Omega_{E}^{\bullet})= \Omega^{\bullet}(A_S, S^lE^*).\]
Hence we still have to show that the two differentials coincide. Let $d_l$ be the differential for the cohomology of $A_S$ with coefficients in
$S^{l}E^*$. Denote
\[\delta:=d^1_S\gamma=\pi^{\mathrm{v}}+\Gamma+\sigma.\]
Since both $d_l$ and $ad_{\delta}$ act as derivations and $d_l$ is the $l$-th symmetric power of
$d_1$, it is enough to prove that for $\omega\in\Omega(S)$, $V\in\Gamma(E)$ and $\psi\in\Gamma(E^*)$, we have that
\begin{equation}\label{Equation_blabla}
d_0(\omega)=[\delta,\omega]_{\ltimes}, \ \ d_0(V)=[\delta,V]_{\ltimes},\ \ d_1(\psi)=[\delta,\psi]_{\ltimes}.
\end{equation}

For the first equation, observe that by Lemma \ref{OmegaS-Central} and (\ref{gamma-represents-d}), we have that
$[\delta,\omega]_{\ltimes}=\mathcal{L}_{\gamma_S}\omega=d\omega$ and since the anchor is a Lie algebroid map, we also have that
$d_0\omega=d\omega$.

In the computations below, we use that, for $W\in\Gamma(\Lambda^{\bullet} E)$ and $\eta\in\Gamma(E^*)$, we have that $W(\eta,\cdot)=-[\eta,W]$,
where on the left hand side we just contract $W$ with $\eta$ and on the right hand side we regard $W$ and $\eta$ as elements in
$\mathfrak{X}^{\bullet}(E)$. Let $V\in\Gamma(E)$. For $\eta,\psi\in\Gamma(E^*)$, using that $C^{\infty}(S)$ commutes with $\Omega_E$ and that
the only term in $[\delta,V]_{\ltimes}$ which is nonzero on $\Gamma(E^*)\times \Gamma(E^*)$ is $[\pi^{\mathrm{v}},V]$, we have:
\begin{eqnarray*}
d_0(V)(\eta,\psi)&=&-V([\eta,\psi]_A)=[[\psi,[\pi^{\mathrm{v}},\eta]],V]=\\
&=&[\psi,[[\pi^{\mathrm{v}},V],\eta]]+[\psi,[\pi^{\mathrm{v}},[\eta,V]]]+[[\psi,V],[\pi^{\mathrm{v}},\eta]]=\\
&=&[\pi^{\mathrm{v}},V](\eta,\psi)-[\psi,[\pi^{\mathrm{v}},V(\eta)]]-[V(\psi),[\pi^{\mathrm{v}},\eta]]=\\
&=&[\pi^{\mathrm{v}},V](\eta,\psi)=[\delta,V]_{\ltimes}(\eta,\psi).
\end{eqnarray*}
For $X\in\mathfrak{X}(S)$ and $\eta\in\Gamma(E^*)$, we have that
\begin{eqnarray*}
d_0(V)(X,\eta)&=&[X,V(\eta)]-V([X,\eta]_A)=[X,[V,\eta]]-[V,[\mathrm{hor}_{\Gamma}(X),\eta]]=\\
&=&[\mathrm{hor}_{\Gamma}(X),[V,\eta]]-[V,[\mathrm{hor}_{\Gamma}(X),\eta]]=[\mathrm{hor}_{\Gamma}(X),V](\eta)=\\
&=&[\Gamma,V]_{\ltimes}(X,\eta)=[\delta,V]_{\ltimes}(X,\eta),
\end{eqnarray*}
where we have used the fact that the only term in $[\delta,V]_{\ltimes}$ which is nonzero on
$\mathfrak{X}(S)\times\Gamma(E^*)$ is $[\Gamma,V]_{\ltimes}$. Consider now $X,Y\in\mathfrak{X}(S)$. We have that
\begin{eqnarray*}
d_0(V)(X,Y)&=&-V([X,Y]_A)=-V(\sigma(X,Y))=-[\sigma(X,Y),V]=\\
&=&[V,\sigma](X,Y)=[\delta,V]_{\ltimes}(X,Y),
\end{eqnarray*}
where we have used the fact that the only term in $[\delta,V]_{\ltimes}$ which is nonzero on $\mathfrak{X}(S)\times \mathfrak{X}(S)$ is $[\sigma,V]_{\ltimes}$.
Hence we have proven the second equation in (\ref{Equation_blabla}).

Let $\psi\in\Gamma(E^*)$. For $\eta\in\Gamma(E^*)$ and respectively $X\in\mathfrak{X}(S)$, we have that
\begin{align*}
d_1(\psi)(\eta)&=-[\psi,\eta]_A=-[\eta,[\pi^{\mathrm{v}},\psi]]=[\pi^{\mathrm{v}},\psi](\eta)=[\delta,\psi]_{\ltimes}(\eta),\\
d_1(\psi)(X)&=[X,\psi]_A=[\mathrm{hor}_{\Gamma}(X),\psi]=[\Gamma,\psi]_{\ltimes}(X)=[\delta,\psi]_{\ltimes}(X),
\end{align*}
and this proves also the last equation in (\ref{Equation_blabla}).
\end{proof}

We now return to the proof of Theorem \ref{theorem-1}. We will use the path of Poisson structures $\pi_t$ provided by Proposition
\ref{the-Moser-path} and the associated Dirac elements $\gamma_{t}$, with $\gamma_{1}= \gamma$ corresponding to $\pi$ and $\gamma_{0}=
j^{1}_{S}\gamma$. The first part of the proof holds for general paths $\gamma_t$. 
We are looking for a family $\mu_{t}$ of diffeomorphisms defined on a neighborhood of $S$ in $E$, for all $t\in [0, 1]$, such that $\mu_{t}=
\textrm{Id}_S$ on $S$, $\mu_0= \textrm{Id}$ and
\begin{equation}
\label{desired-homotopy}
\mu_{t}^{*}\pi_t= j^1_S\pi
\end{equation}
for all $t\in [0,1]$. Then $\mu_1$ will be the desired isomorphism. We will define $\mu_t$ as the flow of a time depending vector field $Z_t$,
i.e.\ as the solution of:
\[ \frac{d}{dt} \mu_t(x)= Z_t(\mu_t(x)), \ \ \mu_0(x)= x.\]
Hence we are looking for the time dependent $Z_t$ defined on a neighborhood of $S$ in $E$. The first condition we require is that $Z= 0$ along
$S$. This implies that $\mu_t= \textrm{Id}$ on $S$ and that $\mu_t$'s are well-defined up to time 1 on an open neighborhood $U$ of $S$, i.e.
$[0, 1]\times U$ is inside the domain $\mathcal{D}$ of the flow of $Z_t$. For the last assertion one uses again the tube lemma: for each $x\in
S$, since $[0, 1]\times \{x\}\subset \mathcal{D}$, one can find an open $U_x\ni x$ with $[0, 1]\times U_x\subset \mathcal{D}$ and one takes $U=
\cup_x U_x$.

Finally, since (\ref{desired-homotopy}) holds at $t=0$, it suffices to require its infinitesimal version: 
\begin{equation}\label{homotopy-equation0}
\mathcal{L}_{Z_t}(\pi_t)+\dot{\pi}_t=0.
\end{equation}
The next step is to rewrite this equation in terms of the Dirac elements $\gamma_{t}$. For each $t$ we consider the isomorphism induced by
$\gamma_t$ (see Proposition \ref{PoissonCohomology}):
\[ \tau_{t}: (\mathfrak{X}^{\bullet}(E),d_{\pi_t})\rmap (\Omega_E^{\bullet},ad_{\gamma_t}).\]
Strictly speaking, the map $\tau_t$ is defined only on the open $N$ from Proposition \ref{the-Moser-path}, but since the discussion is local
around $S$, we allow ourselves this notational sloppiness.  Finding the $Z_t$'s is equivalent to finding
\[ V_t:= \tau_t(Z_t)\in \Omega_{E}^{1}.\]
The following is our compact version of Proposition 2.14 of \cite{Vorobjev2}.

\begin{lemma}\label{HomotopyEquation}
The homotopy equation (\ref{homotopy-equation0}) is equivalent to the following equation:
\begin{equation}
\label{homotopy-equation} [\gamma_t,V_t]_{\ltimes}=\dot{\gamma}_t \
\ (\forall) \ t\in [0, 1]
\end{equation}
required to hold on a neighborhood of $S$ in $E$.
\end{lemma}
\begin{proof}
If we decompose $Z_t=X_t+Y_t$, into its $\Gamma_{\pi_t}$-horizontal and vertical components, then, using the explicit form of $\tau_{t}$ from
(\ref{tau-explicit}), it follows that $V_t$ is the sum of two elements of bi-degree $(1,0)$ and $(0,1)$ respectively:
\[V_t=Y_t-\mathbb{F}_{\pi_t}^{\sharp}(X_t).\]
Therefore (\ref{homotopy-equation}) brakes down into the following list of equations, of various degrees:
\[[\pi_t^{\mathrm{v}},Y_t]_{\ltimes}=\dot{\pi}_t^{\mathrm{v}},\ \ [\Gamma_{\pi_t},Y_t]_{\ltimes}-[\pi_t^{\mathrm{v}},\mathbb{F}_{\pi_t}^{\sharp}(X_t)]_{\ltimes}=\dot{\Gamma}_{\pi_t},\ \ [\mathbb{F}_{\pi_t},Y_t]_{\ltimes}-[\Gamma_{\gamma_t},\mathbb{F}_{\pi_t}^{\sharp}(X_t)]_{\ltimes}=\dot{\mathbb{F}}_{\pi_t}.\]
These are precisely the equations appearing in Proposition 2.14 of \cite{Vorobjev2}.
\end{proof}

In conclusion, we are looking for elements $V_{t}\in
\Omega_{E}^{1}$, defined for all $t\in [0, 1]$
on some open
neighborhood of $S$ in $E$, with the property that $V_{t}|_{S}= 0$
and satisfying the equations (\ref{homotopy-equation}). There is one equation
for each $t$ but, since $\gamma_t$ is of a special type, one
can reduce everything to a single equation.

\begin{lemma} Assume that there exists $X\in \Omega^{1}_{E}$ such that $j_{S}^{1}X= 0$ and
\begin{equation}\label{EquationV}
[\gamma, X]_{\ltimes}= \dot{\gamma}_{1}.
\end{equation}
Then $V_t:= t^{-1}\varphi_t(X)$ satisfies the homotopy equations (\ref{homotopy-equation}).
\end{lemma}

\begin{proof} The condition that the first jet of $X$ along $S$ vanishes ensures that $V_t$
is a smooth family defined also at $t=0$ and that $V_t$ vanishes along $S$. We check the homotopy equations at all $t\in (0, 1]$. For the left
hand side:
\begin{eqnarray*}
\nonumber[\gamma_t,V_t]_{\ltimes}=[\varphi_t(\gamma)+(1-t^{-1})\omega_S,V_t]_{\ltimes}=\varphi_t([\gamma,\varphi_{t^{-1}}(V_t)]_{\ltimes})=t^{-1}\varphi_t([\gamma, X]_{\ltimes}),
\end{eqnarray*}
where we have used the fact that $\omega_S$ lies in the center of
$\Omega_E$ and that $\varphi_t$ commutes with the brackets. Using
the assumption on $X$, we find
\[ [\gamma_t,V_t]_{\ltimes}=t^{-1}\varphi_t(\dot{\gamma}_{1}).\]
Hence (\ref{homotopy-equation}) will follow for $t\in(0,1]$ if we prove the following equation
\begin{equation}\label{Auxiliary}
t^{-1}\varphi_t(\dot{\gamma}_{1})= \dot{\gamma}_t.
\end{equation}
Note first that the explicit formula for $\gamma_{t}$ implies that
\[\varphi_t(\gamma_s)=\gamma_{ts}+(1-t^{-1})\omega_S.\]
Taking the derivative of this equation at $s=1$, we obtain the result.
\end{proof}
The equation for $X$ in the last lemma lives in the cohomology of $(\Omega_{E}^{\bullet}, \textrm{ad}_{\gamma})$. Note
that the right hand side of the equation is indeed closed. This follows by taking the derivative with respect to $t$ at $t= 1$ in
$[\gamma_t, \gamma_t]_{\ltimes}= 0$. Using Proposition
\ref{PoissonCohomology} and the first assumption of Theorem \ref{theorem-1} we see that, after eventually shrinking its domain of
definition, $\dot{\gamma}_1$ exact. This ensures the existence of $X$. To conclude the
proof of the theorem, we still have to show that $X$ can be corrected so that $j_{S}^{1}X= 0$. We will do so by finding $F\in
\Omega_{E}^{0}=C^{\infty}(E)$, such that
\[ j^{1}_{S}([\gamma, F]_{\ltimes})= j^{1}_{S}(X).\]
Then $X'= X- [\gamma,F]_{\ltimes}$ will be the correction of $X$. Since the condition only depends on $j^{1}_{S}F$, it suffices to look for $F$
of type
\[ F= F_0+ F_1\in \textrm{gr}_{0}(\Omega_{E}^{0})\oplus \textrm{gr}_{1}(\Omega_{E}^{0})= C^{\infty}(S)\oplus \Gamma(E^*).\]
So $F|_{S}= F_0$, $d_{S}^{1}F= F_1$. To compute $j^{1}_{S}([\gamma, F]_{\ltimes})$, we write:
\begin{align*}
\varphi_t([\gamma, F]_{\ltimes})&=[\varphi_t(\gamma), \varphi_t(F)]_{\ltimes}=[-t^{-1}\omega_S+d_{S}^{1}\gamma+ t d_{S}^{2}\gamma+ t^2(\ldots),t^{-1} F_0+ F_1]_{\ltimes}=\\
&=t^{-1}[d_{S}^{1}\gamma, F_0]_{\ltimes}+[d_{S}^{1}\gamma, F_1]_{\ltimes}+t(\ldots),
\end{align*}
where we have used that $\omega_S$ and $F_0$ commute with $\Omega_E$ (Lemma \ref{OmegaS-Central}). So we have to solve the following
cohomological equations:
\begin{equation}\label{F0-F1}
[d_{S}^{1}\gamma, F_0]_{\ltimes}= X|_{S},\ \ \  [d_{S}^{1}\gamma, F_1]_{\ltimes}= d_{S}^{1}X.
\end{equation}
By Lemma \ref{cohom-restricted-algebroid}, the relevant cohomologies are precisely the ones assumed to vanish in the theorem. So 
it is enough to show that $X|_{S}$ and $d_{S}^{1}X$ are closed with respect to the differential
$[d_{S}^{1}\gamma,\cdot]_{\ltimes}$. These are precisely the first order consequences of the equation (\ref{EquationV}) that $X$ satisfies. By
(\ref{Auxiliary}), we have that $\varphi_t(\dot{\gamma}_1)=t\dot{\gamma}_t$, hence $j^1_S(\dot{\gamma}_1)=0$. So the Newton formula of Lemma
\ref{Newton-formula} applied to (\ref{EquationV}) gives:
\begin{equation}\label{Equation_X}
[d_{S}^{1}\gamma, X|_S]_{\ltimes}= 0,\ \  [d_{S}^{1}\gamma, d_{S}^{1}X]_{\ltimes}+ [d_{S}^{2}\gamma, X|_{S}]_{\ltimes}= 0.
\end{equation}
Hence $X|_{S}$ is closed, and so we can find $F_0$ satisfying the first equation in (\ref{F0-F1}). In particular, using also Lemma
\ref{OmegaS-Central} this shows that
\[X|_{S}=[d_{S}^{1}\gamma, F_0]_{\ltimes}=dF_0\in\Omega(S).\]
Hence $X|_{S}$ commutes with $\Omega_E$, and since $d_{S}^{2}\gamma\in\Omega_E$, the second equation of (\ref{Equation_X}) becomes
$[d_{S}^{1}\gamma, d_{S}^{1}X]_{\ltimes}=0$. This finishes the proof.

\begin{remark}\label{remark-corollary}
The previous arguments reveal a certain cohomology class related to the linearization problem, which we will describe now more explicitly. First
of all, $\pi$ itself defines a class $[\pi]\in H^{2}_{\pi}(M)_{S}$.
Consider a tubular neighborhood $p:E\rmap S$. Since any two tubular neighborhoods are isotopic, it follows that the class
\[[p^*(\omega_S)]\in H^{\bullet}_{\mathrm{dR}}(M)_{S}= \lim_{S\subset U} H^{\bullet}_{\mathrm{dR}}(U),\]
is independent of $p$. On the other hand, $\pi$ gives a chain map between the complexes
\[\wedge^{\bullet} \pi_{|U}^{\sharp}: (\Omega^{\bullet}(U),d) \rmap (\mathfrak{X}^{\bullet}(U),d_{\pi}).\]
Let $[\pi|_{S}]\in H^{2}_{\pi}(M)_{S}$ be the image of $[p^*(\omega_S)]$ under the induced map in cohomology.
\begin{definition}
The linearization class associated to an embedded leaf $S$ of a Poisson manifold $(M,\pi)$ is defined by
\[l_{\pi,S}:=[\pi]-[\pi_{|S}]\in H^2_{\pi}(M)_S.\]
\end{definition}
It would be interesting to study this class in more detail.
\begin{corollary} \label{remark-corollary2} In the previous theorem, the condition
$H^{2}_{\pi}(M)_{S}= 0$ can be replaced by the condition $l_{\pi,S}=0$.
\end{corollary}

\begin{proof} In the proof above we only used the vanishing of
$[\dot{\gamma}_{1}]$, hence, by Proposition \ref{PoissonCohomology}, it suffices to show that $[\tau_{\pi}^{-1}(\dot{\gamma}_1)]=l_{\pi,S}$.
By the definition of $\varphi_t$, we have that
\[\gamma_t=(t^{-1}-1)\omega_S+t^{-1}m_t^*(\mathbb{F}_{\pi})+m_t^{*}(\Gamma_{\pi})+tm_t^*(\pi^{\mathrm{v}}).\]
Since $m_{e^t}$ is the flow of the Liouville vector field $\mathcal{E}=\sum y_i\frac{\partial}{\partial y_i}$, we obtain
\[\dot{\gamma}_1=-\omega_S-\mathbb{F}_{\pi}+\pi^{\mathrm{v}}+[\mathcal{E},\gamma_1]_{\ltimes}.\]
Using (\ref{2Form}), (\ref{tau-explicit}) and the fact that $\pi^{\mathrm{v}}$ is vertical, we obtain the conclusion
\[\tau_{\pi}^{-1}(\pi^{\mathrm{v}}-\mathbb{F}_{\pi}-\omega_{S})=
\pi^{\mathrm{v}}+\pi^{h}-\wedge^2\pi^{\sharp}(\omega_S)=\pi-\wedge^2\pi^{\sharp}(\omega_S).\]
\end{proof}
\end{remark}

\section{Proof of the main theorem; step 2: Integrability}
\label{Proof of the main theorem; step 2: Integrability}

In this section we show that the conditions of our main theorem imply the integrability of the Poisson structure around the symplectic leaf
which, in turn, implies that the cohomological conditions from Theorem \ref{theorem-1} are satisfied. As in 
\cite{CrFe-Conn}, this step will be divided into three
sub-steps (the three subsections of this section):
\begin{itemize}
\item Step 2.1: Integrability implies the needed cohomological conditions.
\item Step 2.2: Existence of a ``nice'' symplectic realization implies integrability.
\item Step 2.3: Prove the existence of such ``nice'' symplectic realizations.
\end{itemize}
In the first sub-step we will also finish the proof of Proposition \ref{main-cor}.


\subsection{Step 2.1: Reduction to integrability}
\label{Step 2.1: Reduction to integrability}
As promised:



\begin{theorem}\label{TheoremRedInt}
Let $(M,\pi)$ be a Poisson manifold, $x\in M$, and $S$ be the symplectic leaf through $x$. If $P_x$, the homotopy bundle at $x$, is smooth and
compact, then
\[\ H^{1}_{\pi, S}(M)= 0, \ H^{1}_{\pi, S}(M, \nu_{S}^{*})= 0.\]
If moreover $H^2(P_x)= 0$ and $S$ admits an open neighborhood $U$
whose associated groupoid $\Sigma(U, \pi|_{U})$ is smooth and
Hausdorff, then also
\[ H^{2}_{\pi}(M)_{S}=  0.\]
\end{theorem}

\begin{proof}
The first part of the proof is completely similar to that of Theorem 2 in \cite{CrFe-Conn}: a consequence of the Van Est isomorphism and the
vanishing of differentiable cohomology for proper groupoids. More precisely: the conditions on $P_x$ imply that the groupoid $\mathcal{G}(A_S)$
of $A_S= T^{*}M|_{S}$ is smooth and compact (hence the differentiable cohomology with coefficients vanishes); since its $s$-fibers are
1-connected, the Van Est map with coefficients is an isomorphism in degrees 1 and 2; hence the cohomology of $T^*M|_{S}$ in degrees 1 and 2 with
any coefficients vanishes.

For the second part, let $\Sigma(U)=\Sigma(U, \pi|_{U})$ be the symplectic groupoid integrating $(U,\pi_{|U})$. It suffices to show that, for
any open $W\subset U$ containing $S$, there exists a smaller one $V$ such that $H^{2}_{\pi}(V)= 0$. Proceeding as
in the first part, it suffices to produce $V$'s for which $\Sigma(V)= \Sigma(V, \pi|_{V})$ has $s$-fibers which are compact and cohomologically
2-connected. Let $\mathcal{G}\subset \Sigma(U)$ be the set of arrows with source and target inside $W$ and for which both the $s$-fiber and the
$t$-fiber are diffeomorphic to $P_x$. By local Reeb stability applied to the foliation by the $s$-fibers (and $t$-fibers respectively), we see
that all four conditions are open, therefore $\mathcal{G}\subset\Sigma(U)$ is open, and by assumption, all arrows above $S$ are in
$\mathcal{G}$. By the way $\mathcal{G}$ was defined, we see that it is an open subgroupoid over the invariant open $V:=s(\mathcal{G})$. So
$\mathcal{G}=\Sigma(V)$ and it has all the desired properties.
\end{proof}

\begin{proof}[End of the proof of Proposition \ref{main-cor}]
We will adapt the previous proof, making use of Remark \ref{remark-corollary}. We have to show that
\[ [\pi]- [\pi|_{S}]= 0 \in H^{2}_{\pi}(M)_{S}.\]
We show that, for any tubular neighborhood $p: W\rmap S$ of $S$ with $W\subset U$,
there exists a smaller one $V$ such that
\begin{equation*}
[\pi]- [\pi|_{S}]=0 \in H^{2}_{\pi}(V).
\end{equation*}
Let $V$ be as in the previous proof. We may assume that $V$ is connected; if not, we replace it by the component containing $S$. Since
$\Sigma(V)$ is still proper, it suffices to show that the class above is in the image of the Van Est map of $\Sigma(V)$ (in degree $2$). By
Corollary 2 in \cite{Cra}, this image consists of elements $[\omega]\in H^{2}_{\pi}(V)$ for which $\int_{\gamma}\omega=0$, for all $2$-spheres
$\gamma$ in the $s$-fibers of $\Sigma(V)$. Hence it suffices to show that $[\pi]-[\pi_{|S}]$ satisfies this condition. As before, let
$\widetilde{\omega}_S= p^*\omega_S$; also consider the symplectic form $\Omega$ on the symplectic groupoid. The right invariant 2-form on the
$s$-fibers of $\Sigma(V)$ corresponding to $\widetilde{\omega}_S$ is $t^*(\widetilde{\omega}_{S|V})$ restricted to the $s$-fibers. The right
invariant 2-form on the $s$-fibers of $\Sigma(V)$ corresponding to the class $[\pi]$ is the pullback by $t$ to the $s$-fibers of the symplectic
structure on the leaves of $(V,\pi_{|V})$; on the other hand, it is also the restriction to the $s$-fibers of $\Omega$. In particular the two
coincide on the $s$-fibers above $S$. For the correspondence between Poisson cocycles and right invariant, foliated 2-forms the symplectic
groupoid, see \cite{WeinXu}. Consider
\[\omega:=\Omega-t^*(\widetilde{\omega}_{S|V}).\]
Let $\gamma$ a 2-sphere in an $s$-fiber of $\Sigma(V)$. Since $V$ is connected, we can find a homotopy between $\gamma$ and a 2-sphere
$\gamma_1$ which lies in an $s$-fiber over $S$. Since $\omega$ is closed, we have that $\int_{\gamma}\omega=\int_{\gamma_1}\omega$, and since
the restriction of $\omega$ to $s$-fibers over $S$ vanishes, it follows that $\int_{\gamma_1}\omega=0$. This implies the conclusion.
\end{proof}

\subsection{Step 2.2: Reduction to the existence of "nice" symplectic realizations}

Next, we show that the integrability condition required in the last theorem is implied by
the existence of a symplectic realization with some specific properties.

We will use the following notation. Given a symplectic realization $\mu$, we denote by $\mathcal{F}(\mu)$ the foliation defined by $\mu$,
identified also with the involutive distribution $\textrm{Ker}(d\mu)$. Its symplectic orthogonal is a new
distribution $\mathcal{F}(\mu)^{\perp}$. Since $\mu$ is a
Poisson map, it is well-known (and follows easily) that $\mathcal{F}(\mu)^{\perp}$ is also involutive.

\begin{theorem}\label{theorem-step-2.2} Let $(M, \pi)$ be a Poisson manifold and let $S$ be a symplectic leaf. Assume that there exists a symplectic realization
\[ \mu: (\Sigma, \Omega)\rmap (U, \pi|_{U})\]
of some open neighborhood $U$ of $S$ in $M$ such that any leaf of
the foliation $\mathcal{F}^{\perp}(\mu)$ which intersects
$\mu^{-1}(S)$ is compact and 1-connected.

Then there exists an open neighborhood $V\subset U$ of $S$ such that
the Weinstein groupoid $\Sigma(V, \pi|_{V})$ is Hausdorff and
smooth.
\end{theorem}

\begin{proof}
We may assume that all leaves of $\mathcal{F}(\mu)^{\perp}$ are compact and 1-connected. Otherwise, we replace $\Sigma$ by $\Sigma'$
and $U$ by $U'= \mu(\Sigma')$, where $\Sigma'$ is defined as the set of points $y\in\Sigma$ with the property that the leaf of
$\mathcal{F}(\mu)^{\perp}$ through $y$ is compact and 1-connected. Local Reeb stability implies that $\Sigma'$ is open in $\Sigma$. The
hypothesis implies that $\mu^{-1}(S)\subset \Sigma'$ and, since $\mu$ is open, $U'$ is an open neighborhood of $S$.

Clearly, we may also assume that $U= M$. Hence we have a symplectic realization
\[ \mu: (\Sigma, \Omega)\rmap (M, \pi)\]
with the property that all the leaves of $\mathcal{F}(\mu)^{\perp}$ are compact and 1-connected. We claim that $\Sigma(M, \pi)$ has the
desired properties. By Theorem 8 in \cite{CrFe2}, if the symplectic realization $\mu$ is complete, then $\Sigma(M, \pi)$ is smooth. The
compactness assumption on the leaves of $\mathcal{F}(\mu)^{\perp}$ implies that $\mu$ is complete since the Hamiltonian vector fields of type
$X_{\mu^*(f)}$ are tangent to these leaves. For Hausdorffness, we take a closer look to the argument of \cite{CrFe2}. It is based on a natural
isomorphism of groupoids
\[ \Sigma(M, \pi)\times_{M}\Sigma \cong \mathcal{G}(\mathcal{F}(\mu)^{\perp}),\]
where the left hand side is the fibered product over $s$ and $\mu$, and the right hand side is the homotopy groupoid of the foliation
$\mathcal{F}(\mu)^{\perp}$. 
Since homotopy groupoids are always smooth, \cite{CrFe2} concluded that $\Sigma(M, \pi)$ is smooth. In our case, the leaves of
$\mathcal{F}(\mu)^{\perp}$ are 1-connected, hence the homotopy groupoid is a subgroupoid of $M\times M$. So it is Hausdorff, from which it
follows easily that also $\Sigma(M, \pi)$ is Hausdorff.
\end{proof}

\subsection{Constructing symplectic realizations from transversals in the manifold of cotangent paths}
\label{Constructing symplectic realizations from transversals in the manifold of cotangent paths}

In this subsection we describe a general method for constructing symplectic realizations; it will be used in the next subsection to
produce a symplectic realization with the properties required to apply Theorem \ref{theorem-step-2.2}.

In this and the next subsection we will use the same notations as in the proofs of the main results of \cite{CrFe1,CrFe2,CrFe-Conn} and also
some familiarity with those might be useful.

Throughout this subsection, $(M, \pi)$ is an arbitrary Poisson manifold. We know that if $(M, \pi)$ is integrable, then the source map of
$\Sigma(M, \pi)$ produces a complete symplectic realization. It may be helpful to have in mind that, although $(M, \pi)$ may fail to be
integrable, i.e. $\Sigma(M, \pi)$ may fail to be smooth, there is always a ``local groupoid'' $\Sigma_{\textrm{loc}}(M, \pi)$ which is smooth
and produces a symplectic realization of $(M, \pi)$ (but fails to be complete). The plan is to analyze closer the explicit construction of
$\Sigma(M, \pi)$ and of its symplectic form to produce other symplectic realizations, sitting in between $\Sigma(M, \pi)$ and
$\Sigma_{\textrm{loc}}(M, \pi)$. They will have a better chance of both being smooth and having the desired properties.

We will use the same notations as in \cite{CrFe1,CrFe2}. We consider:
\begin{itemize}
\item $\widetilde{\mathcal{X}}= \widetilde{P}(T^*M)$ is the space of all $C^2$-paths in $T^*M$. Recall \cite{CrFe1} that
$\widetilde{\mathcal{X}}$ has a natural structure of Banach
manifold.
\item $\mathcal{X}= P(T^*M)$ is the space of all cotangent paths which, by Lemma 4.6 in \cite{CrFe1}, is a Banach submanifold of $\widetilde{\mathcal{X}}$.
\item $\mathcal{F}= \mathcal{F}(T^*M)$ is the foliation on $\mathcal{X}$ given by the equivalence relation of cotangent homotopy;
it is a smooth foliation on $\mathcal{X}$ of finite codimension.
 In (\ref{the-fol}) we will recall the
description of $\mathcal{F}$ via involutive distributions.
\end{itemize}
Thinking of $\widetilde{\mathcal{X}}$ as the cotangent space of the space $P(M)$ of paths in $M$, it comes with a canonical symplectic structure
$\widetilde{\Omega}$. To avoid issues regarding symplectic structures on Banach manifolds let us just define $\widetilde{\Omega}$ explicitly:
\[\widetilde{\Omega}(X_a,Y_a)=\int_0^1 \omega_{\mathrm{can}}(X_a,Y_a)_{a(t)}dt,\textrm{ for } a\in \widetilde{\mathcal{X}}, X,Y\in T_a\widetilde{\mathcal{X}},\]
where $X_a, Y_a$ are interpreted as paths in $T(T^*M)$ sitting above $a$ and where $\omega_{\mathrm{can}}$ is the canonical symplectic form on
$T^*M$. It can be checked directly that $\widetilde{\Omega}$ is closed; we only need its restriction to $\mathcal{X}$:
\[ \Omega:= \widetilde{\Omega}|_{\mathcal{X}}\in \Omega^2(\mathcal{X}).\]
We will prove that the kernel of $\Omega$ is precisely $\mathcal{F}$ and $\Omega$ is
invariant under the holonomy of $\mathcal{F}$; this ensures that
$\Omega$ descends to a symplectic form on the leaf space $\Sigma(M, \pi)$ (whenever smooth).
Our strategy
is a variation of this idea: we look at transversals $T$ of the foliation and
equivalence relations $\sim$ on $T$ which are weaker than the holonomy; by the same arguments,
$\Omega$ descends to a symplectic form on $T/\sim$, provided this quotient is
smooth. Our job will be to produce $\sim$.\\

We fix a torsion-free connection $\nabla$ on $M$. Then a tangent vector $X$ to $T^*M$ can be interpreted as a pair $(\overline{X}, \theta_X)$,
where $\overline{X}= (dp)(X)\in TM$ and $\theta_X\in T^*M$ is the vertical component with respect to $\nabla$. For torsion-free connections, the
horizontal distribution on $T^*M$ is Lagrangian with respect to $\omega_{\mathrm{can}}$; it follows that:
\begin{equation*}
\omega_{\mathrm{can}}(X,Y)=\langle \theta_Y,\overline{X}\rangle-\langle \theta_X,
\overline{Y}\rangle,
\end{equation*}
Similarly, a tangent vector $X\in T_a\widetilde{\mathcal{X}}$ is represented by a pair $(\overline{X}, \theta_X)$, where $\overline{X}$ is a
$C^1$-path in $TM$, $\theta_X$ in $T^*M$, both sitting above the base path $\gamma= p\circ a$; also,
\begin{equation*}
\widetilde{\Omega}(X,Y)=\int_0^1(\langle \theta_Y,\overline{X}\rangle-\langle \theta_X,
\overline{Y}\rangle)dt.
\end{equation*}

To describe $T\mathcal{X}$ using $\nabla$, one uses two $T^*M$-connections: one on $T^*M$ and one on $TM$, both denoted
$\overline{\nabla}$: for $\alpha,\beta\in\Omega(M)$ and $X\in\mathfrak{X}(M)$,
\[\overline{\nabla}_{\alpha}(\beta)=\nabla_{\pi^{\sharp}(\beta)}(\alpha)+[\alpha,\beta]_{\pi},\quad \overline{\nabla}_{\alpha}(X)=\pi^{\sharp}(\nabla_{X}(\alpha))+[\pi^{\sharp}(\alpha),X].\]
Recall that $[\cdot,\cdot]_{\pi}$ is given by (\ref{bracket_cotangent}). Note that the two connections are
related by
\begin{equation}\label{conn-pi-rel}
\overline{\nabla}_{\alpha}(\pi^{\sharp}(\beta))=\pi^{\sharp}(\overline{\nabla}_{\alpha}(\beta)).
\end{equation}
Since $\nabla$ is torsion-free, it follows that they also satisfy the duality relation:
\begin{equation}\label{ConjConn}
\langle\overline{\nabla}_{\alpha}(\beta),X\rangle+\langle\beta,\overline{\nabla}_{\alpha}(X)\rangle=\pi^{\sharp}(\alpha)(\langle\beta,X\rangle).
\end{equation}
Given a cotangent path $a$ with base path $\gamma$ and a $C^2$-path $U$ in $T^*M$ or $TM$ above
$\gamma$, one has the induced derivative $\overline{\nabla}_{a}(U)$ of $U$ along $a$- a $C^1$-path above $\gamma$, sitting in
the same space as $U$ ($T^*M$ or $TM$). Explicitly,
choosing a time depending section $\tilde{U}$ of class $C^2$ such that $\tilde{U}_t(\gamma(t))=U(t)$,
\[\overline{\nabla}_{a}(U)(x)=\nabla_a\tilde{U}_t(x)+\frac{d \tilde{U}_t}{dt}(x),\textrm{ at }x=\gamma(t).\]
With these, the tangent space
\[ T_a\mathcal{X}\subset T_a\widetilde{\mathcal{X}}\]
corresponds to those pairs $(\overline{X},\theta_X)$ satisfying (see \cite{CrFe1}):
\[\overline{\nabla}_a(\overline{X})=\pi^{\sharp}(\theta_X).\]
Note that, the condition that $\overline{X}$ and $\theta_X$ are of class $C^1$, together with the equation above, forces $\overline{X}$ to be of
class $C^2$. Using equation (\ref{ConjConn}), it is straightforward to show that for $a\in\mathcal{X}$, the two derivatives
$\overline{\nabla}_a$ on $TM$ and $T^*M$ are related by:
\begin{equation}\label{ConjConn2}
\langle\overline{\nabla}_{a}(\theta),V\rangle+\langle
\theta,\overline{\nabla}_{a}(V)\rangle=\frac{d}{dt}\langle \theta,V\rangle,
\end{equation}
for all paths $\theta$ in $T^*M$ and $V$ in $TM$, both sitting over $\gamma=p\circ a$.

To finally define the distribution
$\mathcal{F}\subset T\mathcal{X}$, let $a\in \mathcal{X}$ with base path $\gamma$ and let $\mathcal{E}_{\gamma}$ be the space of all paths
$\beta$ in $T^*M$ of class $C^2$ with base path $\gamma$. Each such path induces a tangent vector in $T_a\mathcal{X}$, with components given by
\[ X_{\beta}:= (\pi^{\sharp}(\beta),\overline{\nabla}_a(\beta))\in T_a\mathcal{X}.\]
With these, the foliation $\mathcal{F}$ can be described as follows
(see \cite{CrFe1})
\begin{equation}\label{the-fol}
\mathcal{F}_{a}= \{ X_{\beta}: \beta\in \mathcal{E}_{\gamma}, \beta(0)=0, \beta(1)=0\}.
\end{equation}

Next, we give a very useful formula for $\Omega$.

\begin{lemma}\label{Omega_Formula_simpla}
Let $a\in \mathcal{X}$ with base path $\gamma$. For $X=(\overline{X},\theta_X)$, $Y=(\overline{Y},\theta_Y)\in T_a\mathcal{X}$ choose
$\beta_X,\beta_Y\in \mathcal{E}_{\gamma}$ such that $\theta_X=\overline{\nabla}_a(\beta_X)$ and $\theta_Y=\overline{\nabla}_a(\beta_Y)$. Then
\[\Omega(X,Y)=\langle \beta_Y,\overline{X}\rangle|_0^1-\langle \beta_X,\overline{Y}\rangle|_0^1-\pi(\beta_X,\beta_Y)|_0^1.\]
In particular, for $Y=X_{\beta}$, with $\beta\in\mathcal{E}_{\gamma}$ we have that $\Omega(X,X_\beta)=\langle\beta,\overline{X}\rangle|_0^1$.
\end{lemma}
\begin{proof}
Using formulas (\ref{ConjConn2}) and (\ref{conn-pi-rel}), we compute:
\begin{eqnarray*}
\Omega(X,Y)=\int_0^1(\langle \theta_Y,\overline{X}\rangle-\langle
\theta_X,\overline{Y}\rangle)dt=\int_0^1(\langle
\overline{\nabla}_a(\beta_Y),\overline{X}\rangle-\langle
\overline{\nabla}_a(\beta_X),\overline{Y}\rangle)dt=\\
=\int_0^1\frac{d}{dt}(\langle \beta_Y,\overline{X}\rangle-\langle \beta_X,\overline{Y}\rangle)dt-\int_0^1(\langle \beta_Y,\overline{\nabla}_a(\overline{X})\rangle-\langle
\beta_X,\overline{\nabla}_a(\overline{Y})\rangle)dt=\\
=\langle \beta_Y,\overline{X}\rangle|_0^1-\langle
\beta_X,\overline{Y},\rangle|_0^1-\int_0^1(\langle
\beta_Y,\pi^{\sharp}(\theta_X)\rangle-\langle \beta_X,\pi^{\sharp}(\theta_Y)\rangle)dt,
\end{eqnarray*}
\begin{eqnarray*}
\int_0^1\langle \beta_Y,\pi^{\sharp}(\theta_X)\rangle dt=\int_0^1\langle
\beta_Y,\pi^{\sharp}(\overline{\nabla}_a(\beta_X))\rangle dt=\int_0^1\langle
\beta_Y,\overline{\nabla}_a(\pi^{\sharp}(\beta_X))\rangle dt=\\
=-\int_0^1\langle \overline{\nabla}_a(\beta_Y),\pi^{\sharp}(\beta_X)\rangle
dt+\int_0^1\frac{d}{dt}(\langle
\beta_Y,\pi^{\sharp}(\beta_X)\rangle)dt=\\
=-\int_0^1\langle \theta_Y,\pi^{\sharp}(\beta_X)\rangle dt+\langle
\beta_Y,\pi^{\sharp}(\beta_X)\rangle|_0^1=\int_0^1\langle \beta_X,\pi^{\sharp}(\theta_Y)\rangle dt+\pi(\beta_X,\beta_Y)|_0^1.
\end{eqnarray*}
\end{proof}

\begin{corollary}\label{CoroKerOmega}
Let $a\in \mathcal{X}$ with base path $\gamma$. Then
\[\ker(\Omega_{a})=\mathcal{F}_a=\{X_\beta: \beta\in \mathcal{E}_{\gamma}, \beta(0)=0, \beta(1)=0\}.\]
\end{corollary}
\begin{proof}
Consider $X=(\overline{X},\theta_X)\in \ker(\Omega_{a})$. It follows that for all $\xi\in\mathcal{E}_{\gamma}$ we have that
$\Omega_a(X,X_{\xi})=0$, hence by the previous lemma $\overline{X}(0)=0$ and $\overline{X}(1)=0$. Let $\beta\in \mathcal{E}_{\gamma}$ be the
unique solution to the equation $\theta_X=\overline{\nabla}_a(\beta)$ with $\beta(0)=0$. Observe that by (\ref{conn-pi-rel}), both
$\overline{X}$ and $\pi^{\sharp}(\beta)$ satisfy the equation
\[\overline{\nabla}_a(Z)=\pi^{\sharp}(\theta_X),\quad Z(0)=0.\]
Therefore they must be equal, and thus $X=X_\beta$. So, again by the lemma, for all $Y=(\overline{Y},\theta_Y)\in T_a\mathcal{X}$, we have that
$\langle\overline{Y}(1),\beta(1)\rangle=0$. On the other hand, $Y(1)$ can be choose arbitrarily (see the Lemma \ref{LemmaOrto} below), thus
$\beta(1)=0$ and this shows that $X\in \mathcal{F}_a$. The other inclusion follows directly from Lemma \ref{Omega_Formula_simpla}.
\end{proof}

Consider now the maps $\tilde{s}, \tilde{t}: \mathcal{X}\rmap M$
which assign to a path $a$ the starting (respectively ending) point of its base path $\gamma$.

\begin{lemma}\label{LemmaOrto} $\tilde{s}$ and $\tilde{t}$ are submersions and their fibers
are orthogonal with respect to $\Omega$. More precisely, denoting by $\perp$ the orthogonal with respect to $\Omega$, we have that
\[(\ker d\tilde{s}_a)^{\perp}=\ker d\tilde{t}_a,\quad (\ker d\tilde{t}_a)^{\perp}=\ker d\tilde{s}_a.\]
\end{lemma}

\begin{proof}
To prove the first part, note that 
\[\ker d\tilde{s}_a=\{(\overline{X},\theta_X)\in T_a\mathcal{X}:\overline{X}(0)=0\}, \ \ker d\tilde{t}_a=\{(\overline{X},\theta_X)\in T_a\mathcal{X}:\overline{X}(1)=0\}.\]
Then, for $V_0\in T_{\gamma(0)}M$, the path $\overline{V}$ above $\gamma$ which satisfies
\[\overline{\nabla}_a(\overline{V})=0,\ \ \overline{V}(0)=V_0\]
induces $V=(\overline{V},0)\in T_a\mathcal{X}$ with $d\tilde{s}_a(V)=V_0$; so $\tilde{s}$, and similarly $\tilde{t}$, is a submersion.

For $X=(\overline{X},\theta_X)\in \ker d\tilde{t}_a$, $Y=(\overline{Y},\theta_Y)\in \ker d\tilde{s}_a$, let $\beta_X,\beta_Y\in
\mathcal{E}_{\gamma}$ be so that 
\[\overline{\nabla}_a(\beta_X)=\theta_X,\  \beta_X(1)=0,\ \ \  \overline{\nabla}_a(\beta_Y)=\theta_Y,\  \beta_Y(0)=0.\]
Lemma \ref{Omega_Formula_simpla} implies that $\Omega_a(X,Y)=0$. Conversely, let $X=(\overline{X},\theta_X)\in (\ker
d\tilde{s}_a)^{\perp}$. For all $\xi\in \mathcal{E}_{\gamma}$, such that $\xi(0)=0$ we have that $X_{\xi}\in \ker d\tilde{s}_a$, therefore, by
assumption, $\Omega_a(X,X_{\xi})=0$. Thus, by Lemma \ref{Omega_Formula_simpla} we have that
\[0=\Omega_a(X,X_{\xi})=\langle \xi(1),\overline{X}(1)\rangle.\]
But $\xi(1)$ is arbitrary, hence $\overline{X}(1)=0$, i.e. $X\in \ker d\tilde{t}_a$. So $(\ker
d\tilde{s}_a)^{\perp}=\ker d\tilde{t}_a$.
\end{proof}

We collect the main properties of $\Omega$ that are needed in the next subsection.

\begin{proposition}\label{PropTrnasversal}
Let $\mathcal{T}$ be a transversal to $\mathcal{F}$. Then the
following hold:
\begin{enumerate}
\item[(a)]
$\Omega_{|\mathcal{T}}$ is symplectic and is invariant under the holonomy action of $\mathcal{F}$ on $\mathcal{T}$.
\item[(b)] The sets $U_s=\tilde{s}(\mathcal{T})$ and $U_t=\tilde{t}(\mathcal{T})$ are open in
$M$, and
\begin{eqnarray*}
&&\sigma=\tilde{s}_{|\mathcal{T}}:(\mathcal{T},\Omega_{|\mathcal{T}})\to
(U_s,\pi_{|U_s})\textrm{
is a Poisson map and}\\
&&\tau=\tilde{t}_{|\mathcal{T}}:(\mathcal{T},\Omega_{|\mathcal{T}})\to
(U_t,\pi_{|U_t})\textrm{ is anti-Poisson.}
\end{eqnarray*}
\item[(c)] $\ker(\sigma)^{\perp}=\ker(\tau)$ and
$\ker(\tau)^{\perp}=\ker(\sigma)$.
\end{enumerate}
\end{proposition}

\begin{proof}
Since $\mathcal{F}$ is of finite codimension \cite{CrFe1}, there are no issues regarding the meaning of symplectic forms on our $T$ and
$\Omega|_{\mathcal{T}}$ is clearly symplectic. Actually, the entire (a) is a standard fact about kernels of closed two-forms, at least in the
finite dimensions; it applies to our situation as well: from the construction of holonomy by patching together foliation charts, the second part
is a local issue: given a product $B\times \mathcal{T}$ of a ball $B$ in a Banach space and a finite dimensional manifold $\mathcal{T}$ (for us
a small ball in an Euclidean space) and a closed two-form $\Omega$ on $B\times \mathcal{T}$, if
\[ \ker(\Omega_{x, y})= T_{x}B\times\{0_y\} \subset T_xB\times T_{y}\mathcal{T}, \ \ \forall \ (x, y)\in B\times \mathcal{T},\]
then $\Omega_x= \Omega|_{\{x\}\times \mathcal{T}}\in \Omega^{2}(\mathcal{T})$ does not depend on $x\in B$ (since $\Omega$ is closed).

For part (b) we will prove the statement for $\sigma$, for $\tau$ it follows similarly. Since
\[d\tilde{s}_a:T_a\mathcal{X}=T_a\mathcal{T}\oplus \mathcal{F}_a\to T_{\gamma(0)}M\]
is surjective and $\mathcal{F}_a\subset \ker{d\tilde{s}_a}$, it follows that $\tilde{s}_{|\mathcal{T}}$ is a submersion onto the open
$\tilde{s}(\mathcal{T})=U_s$. To show that $\sigma$ is a symplectic realization, we will describe the Hamiltonian vector fields of
$\sigma^*(f)$, for $f\in C^{\infty}(U_s)$. Consider the vector field on $\mathcal{X}$:
\[\widetilde{H}_{f,a}:=X_{(1-t)df_{\gamma(t)}}=(\overline{\nabla}_a((1-t)df_{\gamma(t)}),(1-t)\pi^{\sharp}(df_{\gamma(t)}))\in T_a\mathcal{X}.\]
Then we have that $d\tilde{s}(\widetilde{H}_{f})=\pi^{\sharp}(df)$, and by Lemma \ref{Omega_Formula_simpla} it also satisfies
\[\Omega(\widetilde{H}_{f},Y)_a=\langle df_{\gamma(0)},\overline{Y}(0)\rangle=d(\widetilde{s}^*f)(Y),\ (\forall)\  Y\in T_a\mathcal{X}.\]
Thus $\Omega(\widetilde{H}_{f},\cdot)=d(\widetilde{s}^*f)$. Decomposing $\widetilde{H}_{f|\mathcal{T}}:=H_f+V_f$, where $H_f$ is tangent to
$\mathcal{T}$ and $V_f$ is tangent to $\mathcal{F}$ and using the fact that $\mathcal{F}=\ker\Omega$, it follows that
\[\Omega_{|\mathcal{T}}(H_f,\cdot)=d(\sigma^*f), \ \ d\sigma(H_f)=\pi^{\sharp}(df).\]
This shows that $\sigma$ is Poisson. Part (c)follows from Lemma \ref{LemmaOrto} and Corollary \ref{CoroKerOmega}.
\end{proof}

\subsection{Step 2.3: the needed symplectic realization}
\label{Step 2.3: the needed symplectic realization}

Back to the main theorem, to finish the proof, we still have to prove the existence of a symplectic realization as in Theorem
\ref{theorem-step-2.2}. We will do that using the methods from the previous subsection; in particular, we keep the same notations. We consider:
\begin{itemize}
\item $\mathcal{Y}= \tilde{s}^{-1}(S)\subset \mathcal{X}$, the submanifold of $\mathcal{X}$ sitting above $S$. Note that this
is the same as the manifold $P(A)$ of $A$-paths of the algebroid $A= A_S$.
\item The restriction of $\mathcal{F}$ to $\mathcal{Y}$, $\mathcal{F}_{\mathcal{Y}}= \mathcal{F}|_{\mathcal{Y}}$. Again, this is the foliation $\mathcal{F}(A)$ associated to the algebroid $A$ \cite{CrFe1}, and $\mathcal{Y}/\mathcal{F}$ is the groupoid $\mathcal{G}(A)$ of $A$. 
\end{itemize}
From the assumptions of the theorem, $\mathcal{G}(A)$ is compact. We denote it by $B$ here.

As in the appendix in \cite{CrFe-Conn}, we will use the following technical lemma:

\begin{proposition}
\label{technical} Let $\mathcal{F}$ be a foliation of finite codimension on a Banach manifold $\mathcal{X}$ and let $\mathcal{Y}\subset
\mathcal{X}$ be a submanifold which is saturated with respect to $\mathcal{F}$ (i.e., each leaf of $\mathcal{F}$ which hits $\mathcal{Y}$ is
contained in $\mathcal{Y}$). Assume that:
\begin{itemize}
\item[(H0)] The holonomy groups of the foliation $\mathcal{F}$ at the points of $\mathcal{Y}$ are trivial.
\item[(H1)] $\mathcal{F}_{\mathcal{Y}}:= \mathcal{F}|_{\mathcal{Y}}$ is induced by a submersion $p: \mathcal{Y}\to B$, with $B$-compact.
\item[(H2)] The fibration $p: \mathcal{Y}\to B$ is locally trivial.
\end{itemize}
Then one can find:
\begin{enumerate}[(i)]
\item a transversal $\mathcal{T}\subset \mathcal{X}$ to the foliation $\mathcal{F}$ such that $\mathcal{T}_{\mathcal{Y}}:= \mathcal{Y}\cap \mathcal{T}$ is a complete transversal to $\mathcal{F}_{\mathcal{Y}}$ (i.e., intersects each leaf of $\mathcal{F}_{\mathcal{Y}}$ at least once).
\item a retraction $r: \mathcal{T}\to \mathcal{T}_{\mathcal{Y}}$.
\item an action of the holonomy of $\mathcal{F}_{\mathcal{Y}}$ on $r: \mathcal{T}\to \mathcal{T}_{\mathcal{Y}}$ along $\mathcal{F}$.
\end{enumerate}
Moreover, the quotient of $\mathcal{T}$ by the action of $\mathcal{F}_{\mathcal{Y}}$ is a smooth (Hausdorff) manifold.
\end{proposition}

In our case, once we make sure that the lemma can be applied, the resulting quotient $\Sigma$ of $\mathcal{T}$ will produce the desired
symplectic realization. This follows from Proposition \ref{PropTrnasversal} and the fact that, by construction, $\sigma^{-1}(S)=
\mathcal{G}(A)\subset \Sigma$.

(H1) is clear since $P_x$ is smooth and compact.
For (H2), we need the following:

\begin{lemma}\label{Localy-triv}
Let $M$ be a finite dimensional manifold, $x_0\in M$, and denote by $\mathrm{Path}(M, x_0)$ the Banach manifold of $C^2$ paths in $M$ starting
at $x_0$. Then
\[ \epsilon: \mathrm{Path}(M, x_0)\rmap M, \ \gamma\mapsto \gamma(1)\]
is a locally trivial fiber bundle.
\end{lemma}

\begin{proof}
Consider $x\in M$. For $U\subset V$ a small enough open neighborhoods of $x\in M$, we will construct a smooth family of diffeomorphisms
\[\phi_{y,t}:M\rmap M, \mathrm{ for }\ y\in U,\  t\in\mathbb{R},\]
such that $\phi_{y,t}$ is supported inside $V$, $\phi_{y,0}=id_M$ and $\phi_{y,1}(x)=y$. Then the required trivialization over $U$ is given by
\[\tau_U:\epsilon^{-1}(x)\times U\rmap \epsilon^{-1}(U),\quad \tau_U(\gamma, y)(t)=\phi_{y,t}(\gamma(t)),\]
with inverse
\[\tau^{-1}_{U}(\gamma)(t)=(\phi^{-1}_{\gamma(1),t}(\gamma(t)),
\gamma(1)).\] The construction of such diffeomorphisms is clearly a local issue, thus we may assume that $M=\mathbb{R}^m$, with $x=0$ and
$U=B_1(0)$, $V=B_2(0)$, the balls of radii 1 and 2 respectively. Consider $f\in C^{\infty}(\mathbb{R}^m)$, supported inside $B_2(0)$, with
$f_{|B_1(0)}=1$. Let $\phi_{y,t}$ be the flow at time $t$ of the compactly supported vector field $X_y:=f\overrightarrow y$, where
$\overrightarrow y$ represents the constant vector field on $\mathbb{R}^m$ corresponding to $y\in B_1(0)$. Then it is easy to see that
$\phi_{y,t}$ satisfies all requirements.
\end{proof}

Next, for a groupoid $\mathcal{G}$ over a manifold $S$, we denote by $\mathrm{Path}^{s}(\mathcal{G}, 1)$ the Banach manifold of $C^2$-paths
$\gamma$ in $\mathcal{G}$ starting at some unit $1_x$ and satisfying $s\circ\gamma=x$. For $\mathcal{G}=\mathcal{G}(A)$, Proposition 1.1 of
\cite{CrFe1} identifies our bundle $p: \mathcal{Y}\rmap B$ with the bundle
\begin{equation}\label{map_epsilon}
\widetilde{\epsilon}: \mathrm{Path}^s(\mathcal{G}, 1)\rmap \mathcal{G}, \ \gamma\mapsto \gamma(1).
\end{equation}
Hence the following  implies (H2).

\begin{lemma} For a source locally trivial Lie groupoid $\mathcal{G}$, the map $\widetilde{\epsilon}$ (\ref{map_epsilon}) is a locally trivial fiber bundle.
\end{lemma}

\begin{proof}
Consider $g_0\in\mathcal{G}$, $x_0=s(g_0)$. Consider a local trivialization of $s$ over $U\ni x_0$, $\tau:s^{-1}(U)\cong U\times s^{-1}(x_0)$.
Since the unit map is transversal to $s$, we may assume that $\tau(1_x)=(x,1_{x_0})$ for all
$x\in U$. Left composing with $\tau$ induces a diffeomorphism $\tau_{*}:\widetilde{\epsilon}^{-1}(s^{-1}(U))\rmap U\times
\mathrm{Path}(s^{-1}(x_0),1_{x_0})$, under which $\widetilde{\epsilon}$ becomes
\[Id\times \epsilon: U\times\mathrm{Path}(s^{-1}(x_0),1_{x_0})\rmap U\times s^{-1}(x_0).\]
Since $\epsilon$ is trivial over $V\subset s^{-1}(x_0)$ around $g_0$, $\widetilde{\epsilon}$ is
trivial over $\tau^{-1}(U\times V)$.
\end{proof}

We still have to check (H0).

\begin{lemma} \label{lemma-Getzler1} For any leaf $\mathcal{L}$ of $\mathcal{F}$ inside $\mathcal{Y}$,
$\pi_1(\mathcal{L})\cong \pi_2(P_x)$.
\end{lemma}

\begin{proof} The foliation $\mathcal{F}_{\mathcal{Y}}$ is given by the fibers of $p:\mathcal{Y}\rmap B$, which, as remarked before, is
isomorphic to the bundle $\widetilde{\epsilon}:\mathrm{Path}^s(\mathcal{G}, 1_y)\rmap \mathcal{G}(A)$. So, a leaf $\mathcal{L}$ will be
identified with $\widetilde{\epsilon}^{-1}(g)$, for some $g\in \mathcal{G}(A)$. Denote by $y:=s(g)$ and $P_y:=s^{-1}(y)$. Then $\mathcal{L}$ is
a fiber of $\epsilon:\mathrm{Path}(P_y,y)\rmap P_y$, which, by Lemma \ref{Localy-triv}, is a locally trivial fiber bundle. Since
$\mathrm{Path}(P_y, 1_y)$ is contractible, by a standard argument, we find that $\pi_1(\mathcal{L})\cong \pi_2(P_y)$. Since $\mathcal{G}(A)$ is
transitive, $P_y$ and $P_x$ are diffeomorphic.
\end{proof}

Of course, if $\pi_2(P_x)$ were assumed to be trivial, then condition (H0) follows automatically. Note that the hypothesis
that $H^2(P_x)= 0$ is equivalent to $\pi_2(P_x)$ being finite; we show that this is enough to ensure triviality of the holonomy groups.

\begin{lemma} \label{lemma-Getzler2} The holonomy group of the foliated manifold $(\mathcal{X}, \mathcal{F})$ is trivial at any point
$a\in \mathcal{Y}$.
\end{lemma}

\begin{proof} Let $a\in \mathcal{Y}$ and let $\Gamma$ be the holonomy group at $a$.
Let $\mathcal{T}$ be a transversal of $(\mathcal{X}, \mathcal{F})$ through $a$. Since $\Gamma$ is finite, $\mathcal{T}$ can be chosen small
enough so that the holonomy transformations $\mathrm{hol}_{u}$ define an action of $\Gamma$ on $\mathcal{T}$. Denote by
$\mathcal{T}_{\mathcal{Y}}:= \mathcal{T}\cap \mathcal{Y}$. Since the holonomy of $(\mathcal{Y}, \mathcal{F}_{\mathcal{Y}})$ is trivial, by
making $\mathcal{T}$ smaller, we may assume:
\begin{enumerate}
\item[C1:] the action of $\Gamma$ on $\mathcal{T}_{\mathcal{Y}}$ is trivial.
\end{enumerate}
Note also that the submersion $\sigma: \mathcal{T}\rmap M$ satisfies:
\begin{enumerate}
\item[C2:] $\sigma$ is $\Gamma$-invariant.
\item[C3:] $\sigma^{-1}(\sigma(a))\subset \mathcal{T}_{\mathcal{Y}}$.
\end{enumerate}
C1 is clear. For C2: $a'$ and $\mathrm{hol}_{u}(a')$ are always in the same leaf, i.e. cotangent-homotopic, for $a'\in \mathcal{T}$, hence the starting point
is the same. We have to show that the action of $\Gamma$ is trivial in a neighborhood of $a$ in $\mathcal{T}$. Since
$\Gamma$ is finite, it suffices to show that the induced infinitesimal action of $\Gamma$ on $T_{a}\mathcal{T}$ is trivial. 
At the infinitesimal level, we have a short exact sequence
\[ \ker(d\sigma)_a\stackrel{i}{\rmap} T_{a}\mathcal{T} \stackrel{(d\sigma)_a}{\rmap} T_xM.\]
This is a sequence of $\Gamma$-modules, where $\Gamma$ acts trivially on the first and the last term. For the last map, this follows from C2.
For the first map, C3 implies that $\ker(d\sigma)_a\subset T_a\mathcal{T}_{\mathcal{Y}}$ on which $\Gamma$ acts trivially by C1. Since $\Gamma$
is finite, the action on the middle term must be trivial as well (use e.g. an equivariant splitting).
\end{proof}

\bibliographystyle{amsplain}
\def\lllll{}

\end{document}